\documentclass[11pt,a4paper]{article}
\usepackage{pgfplots}
\pgfplotsset{width=7cm,compat=1.8}
\usetikzlibrary{patterns}
\usepackage[latin1]{inputenc}
\usepackage{latexsym}
\usepackage{amsmath,amssymb,amsfonts,amsthm,bbm}
\usepackage{xcolor}
\usepackage{secdot}
\usepackage{mathrsfs}

\DeclareFontFamily{OT1}{pzc}{}
\DeclareFontShape{OT1}{pzc}{m}{it}{<-> s * [1.10] pzcmi7t}{}
\DeclareMathAlphabet{\mathpzc}{OT1}{pzc}{m}{it}
\usepackage{comment}
\usepackage{bbm}
\usepackage{bm}
\usepackage{paralist}
\usepackage{mhequ}
\usepackage[shortlabels]{enumitem}

\usepackage[
bookmarks=true,
bookmarksnumbered=true,
colorlinks=true, pdfstartview=FitV, linkcolor=blue, citecolor=blue,hypertexnames=false,
urlcolor=blue]{hyperref}

\usepackage[nameinlink,capitalize]{cleveref}

\usepackage{tikz}
\usetikzlibrary{calc}
\usetikzlibrary{decorations.pathreplacing}
\usetikzlibrary{arrows}
\usepgfplotslibrary{fillbetween}
\usepackage{mathtools}

\usepackage[toc,page]{appendix}
\usepackage{enumitem}

\usepackage[numbers,comma,sort]{natbib}

\definecolor{db}{RGB}{0, 0, 130}

\usepackage{pdfsync}

\def\E{\hskip.15ex\mathsf{E}\hskip.10ex}

\definecolor{rp}{rgb}{0.25, 0, 0.75}
\definecolor{dg}{rgb}{0, 0.5, 0}
\newcommand{\A}{\mathcal{A}}
\newcommand{\C}{\mathcal{C}}
\newcommand{\cC}{\mathcal{C}}
\newcommand{\cQ}{\mathcal{Q}}
\newcommand{\Op}{A}

\newcommand{\R}{\mathbb{R}}

\newcommand{\eps}{\varepsilon}

\newcommand{\F}{\mathcal{F}}
\renewcommand{\H}{\mathcal{H}}
\newcommand{\I}{\mathbbm{1}}

\newcommand{\N}{\mathbb{N}}

\newcommand{\EE}{\mathbb{E}}

\newcommand{\PP}{\mathbb{P}}

\newcommand{\cL}{\mathcal{L}}

\newcommand{\bqn}{\begin{equation}}
\newcommand{\eqn}{\end{equation}}
\newcommand{\bqne}{\begin{equation*}}
\newcommand{\eqne}{\end{equation*}}

\renewcommand{\phi}{\varphi}

\newcommand{\wt}{\widetilde}
\renewcommand{\ge}{\geqslant}
\renewcommand{\le}{\leqslant}

\newcommand{\CM}{\H_\mu}

\newcommand{\nn}{\nonumber}

\numberwithin{equation}{section}

\DeclareMathOperator{\law}{Law}%

\usepackage{todonotes}

\makeatletter
\newcommand{\customlabel}[2]{%
   \protected@write \@auxout {}{\string\newlabel {#1}{{#2}{\thepage}{#2}{#1}{}}}%
   \hypertarget{#1}{#2\hspace{-0.14cm}}
}
\makeatother

\usepackage{crossreftools}
\theoremstyle{definition}
\newtheorem{definition}{Definition}[section]

\theoremstyle{remark}
\newtheorem{remark}[definition]{Remark}

\theoremstyle{plain}
\newtheorem{corollary}[definition]{Corollary}
\newtheorem{lemma}[definition]{Lemma}
\newtheorem{proposition}[definition]{Proposition}
\newtheorem{theorem}[definition]{Theorem}
\newtheorem{assumption}[definition]{Assumption}

\crefname{definition}{Definition}{Definitions}
\crefname{notation}{Notation}{Notations}
\crefname{hypothesis}{Hypothesis}{Hypotheses}
\crefname{remark}{Remark}{Remarks}
\crefname{example}{Example}{Examples}
\crefname{corollary}{Corollary}{Corollaries}
\crefname{lemma}{Lemma}{Lemmas}
\crefname{proposition}{Proposition}{Propositions}
\crefname{theorem}{Theorem}{Theorems}
\crefname{assumption}{Assumption}{Assumptions}
\crefname{conjecture}{Conjecture}{Conjectures}

\author{Lukas Anzeletti\footnote{TU Wien ASC, Wiedner Hauptstrasse 8-10, 1040 Wien;   \texttt{lukas.anzeletti@asc.tuwien.ac.at}.} \and \setcounter{footnote}{3}
Oleg Butkovsky\footnote{Weierstrass Institute (WIAS), Berlin; 
Institut f\"ur Mathematik, Humboldt-Universit\"at zu Berlin; Simons Laufer Mathematical Sciences Institute (MSRI), Berkeley;~   \texttt{oleg.butkovskiy@gmail.com}.} \and
M\'at\'e Gerencs\'er\footnote{TU Wien ASC, Wiedner Hauptstrasse 8-10, 1040 Wien;~   \texttt{mate.gerencser@asc.tuwien.ac.at}.}\and
Alexander Shaposhnikov\footnote{\texttt{shal1t7@mail.ru}}
}

\title{Uniqueness for stochastic differential equations in Hilbert spaces with irregular drift}
\begin{document}
\date{February 13, 2026}
\maketitle

\begin{abstract}
We  present a versatile  framework to study strong existence and uniqueness for stochastic differential equations (SDEs) in Hilbert spaces with irregular drift. We consider an SDE in a separable Hilbert space $H$,
\begin{equation*}
dX_t= (\Op X_t + b(X_t))dt +(-\Op)^{-\gamma/2}dW_t,\quad X_0=x_0 \in H,
\end{equation*}
where $A$ is a self-adjoint negative definite operator with purely atomic spectrum, $W$ is a cylindrical Wiener process, $b$ is $\alpha$-H\"older continuous
function $H\to H$, and a nonnegative parameter $\gamma$  such that the stochastic convolution takes values in $H$. We show that this equation has a unique strong solution  provided that  $\alpha > \alpha^*(\gamma)$, with an explicit function $\alpha^*$ that takes values in $(0,1)$ for all $\gamma\in[0,3)$. This substantially extends the seminal work of Da Prato and Flandoli (2010) as no structural assumption on $b$ is imposed. The range of admissible $\alpha$ is also extended. To obtain this result, we do not use infinite-dimensional Kolmogorov equations but instead develop a new technique combining L\^e's theory of stochastic sewing in Hilbert spaces, Gaussian analysis, and a method of Lasry and Lions for approximation in Hilbert spaces.
\end{abstract}

\setcounter{tocdepth}{2}
\renewcommand\contentsname{}
\vspace{-1cm}

\tableofcontents

\section{Introduction}

While regularization by noise is well understood in finite-dimensional stochastic differential equations, much less is known in the infinite-dimensional setting. In this article, we combine  L\^e's stochastic sewing theory \cite{Leinitial,Le} with an infinite-dimensional approximation technique due to Lasry and Lions \cite{LasryLions} to establish strong well-posedness of stochastic evolution equations of the form
\begin{equation} \label{eq:mainSDE}
    dX_t= (\Op X_t + b(X_t))dt +(-\Op)^{-\gamma/2}dW_t,\quad X_0=x_0 \in H,
\end{equation}
see \cref{thm:main}.
Here $(H,\langle\cdot,\cdot\rangle_H)$ is a separable Hilbert space, $A:D(A)\subset H\to H$ is a self-adjoint negative definite operator with purely atomic spectrum, 
$b \in \cC^\alpha(H,H)$ for some $\alpha\in(0,1]$,
$W$ is an $H$-cylindrical Brownian motion on a complete filtered probability space
$(\Omega,\mathcal{F},(\mathcal{F}_t),\mathbb{P})$,
and $\gamma\geqslant 0$ is a parameter chosen so that the stochastic convolution $Z$ (a solution to the linearized version of \eqref{eq:mainSDE} with $x_0=0$, see \eqref{eq:Z}) takes values in $H$, see \cref{ass:lambdagamma} below.

As a corollary, we obtain regularization by noise for the stochastic heat equation with a general irregular drift. Namely, choosing in the above setting $H=L^2([0,1]^d)$, $d\in\N$, the operator $A$ to be the Dirichlet Laplacian, we see that equation \eqref{eq:mainSDE} reads as
\begin{equ}\label{eq:main-app}
dX_t= \Delta X_t  dt + b(X_t)dt +(-\Delta)^{-\gamma/2}dW_t,\quad X_0=x_0.
\end{equ}
For $\gamma\ge0$ with $\gamma>d/2-1$ and $1\leqslant d \leqslant 7$, \cref{cor:maincorollary} identifies a regime of $\alpha$ so that this equation has a unique strong solution.
This significantly improves upon the classical paper of Da Prato and Flandoli \cite{FlandoliDaPrato}, as well as more recent works \cite{AddonaBignamini2023,AddonaBignamini2024,AddonaMasieroPriola}, which impose substantial additional restrictions on the drift $b$ whenever $d \ge 2$.

Indeed, let us recall that Da Prato and Flandoli \cite{FlandoliDaPrato} show that equation \eqref{eq:main-app} has a unique strong solution under, among others, the assumption
\begin{equation}\label{daprf}
\sum_{n=1}^\infty \frac1{\lambda_n} \|\langle b, e_n\rangle_H\|_{\C^\alpha(H,\R)}^2<\infty,
\end{equation}
where $(\lambda_n)$ is a sequence of eigenvalues of $-\Delta$. Since $\|\langle b, e_n\rangle_H\|_{\C^\alpha(H,\R)}\le \|b\|_{\C^\alpha(H,H)}<\infty$ and $\lambda_n\sim n^{2/d}$. this condition is harmless for $d=1$. However, for $d \ge 2$, the eigenvalues $\lambda_n$ do not decay fast enough, so \eqref{daprf} can only be satisfied for very special drifts $b$. The paper \cite{FlandoliDaPrato} uses an infinite-dimensional version of the classical Zvonkin transformation method  \cite{ver80,zvonkin74}, and this condition seems unavoidable when bounding the second derivative of the Zvonkin transform.

Certain attempts to relax condition \eqref{daprf} were made in a number of subsequent papers \cite{AddonaBignamini2023,AddonaBignamini2024,AddonaMasieroPriola}.  
However,  for the stochastic heat equation \eqref{eq:main-app}, these papers still require a certain decay in $n$ of the H\"older norms $   
\|\langle b, e_n\rangle_H\|^2_{\C^\alpha(H,\R)}$ as long as $d\ge2$. These articles also rely on extensions of the Zvonkin method.

A main achievement of the present article is that we completely remove any additional structural assumptions on the drift. We impose only that $b$ is H\"older continuous, see \cref{thm:main}. A notable feature of our approach is that the toolbox is entirely different from that of the existing literature on stochastic evolution equations with irregular drift \cite{FlandoliDaPrato,DaPratoFlandoli2024,DFPR13,AddonaBignamini2023,AddonaMasieroPriola,AddonaBignamini2024,Priola,Priolaerratum}. In particular, our method does not rely on the Zvonkin transformation and thus avoids the associated technical difficulties. Instead, we use stochastic sewing in Hilbert spaces, a tool introduced by L\^e \cite{Leinitial, Le}.

The stochastic sewing strategy for \eqref{eq:mainSDE} in infinite dimensions runs into a few problems.
One basic Gaussian estimate used in \cite{Leinitial} and many following works is 
\begin{equation*}
\big|\EE\big(f(x+Z)-f(Z)\big)\big|\lesssim \|f\|_{L^\infty} \|\cQ^{-1/2}x\|,    
\end{equation*} where $Z\sim \mathcal{N}(0,\cQ)$. In other words, if $\cQ$ is nondegenerate and the dimension is finite, one bounds the Lipschitz norm of $\EE f(\cdot+Z)$ by the supremum norm of $f$.
In infinite dimensions, $\|\cQ^{-1/2}x\|=\infty$ unless $x$ lies in the Cameron-Martin space.
Ensuring that in all applications of this bound the perturbations lie in the Cameron-Martin space necessitates additional regularization and leads to different power counting in infinite dimensions.
Once such a Lipschitz estimate is obtained, another key difficulty is how to properly extend the classical finite-dimensional theory of interpolation to the infinite-dimensional context. Indeed, if we want to prove a result about a H\"older continuous function $f\colon \R^d \to \R^d$, we can decompose it into a sum of a Lipschitz function $P_\lambda f$ and a bounded function $f - P_\lambda f$, where $P_\lambda$ is a Gaussian  kernel with variance $\lambda I_d$. We then apply the result separately to the Lipschitz and bounded parts, obtaining a free parameter $\lambda$. By tuning $\lambda$, we can then obtain the desired result for $f$. Performing this procedure directly in infinite dimensions is impossible, since $P_\lambda$ is not defined. Instead, we overcome this problem by using an approximation method developed by Lasry and Lions \cite{LasryLions}. Lastly, in order to optimize the regime of $\gamma$ and $\alpha$, we measure time regularity of the drift term both as a process taking values in $H$ as well as considering it as a process with values in a Sobolev subspace thereof. The regime is then identified by finetuning the precise regularity of the Sobolev space.

There are various aspects that make the stochastic sewing approach appealing. First, a natural next step of the theory would be, in the spirit of Krylov-R\"ockner \cite{KrylovRoeckner}, to allow time-dependence in the drift, i.e. replacing $b(X_t)$ by $b(t,X_t)$ in \eqref{eq:mainSDE} and assuming e.g. $b\in L^q([0,1],\cC^\alpha(H,H))$ with appropriate conditions on $q$ and $\alpha$.
Based on \cite{GaleatiGerencser,butkovskygallay}
the tools to handle such time-dependence in the stochastic sewing framework are readily available.
The treatment of time-dependent drift via the Kolmogorov equations at the very least requires moving from the elliptic to parabolic theory, which could be technically demanding. It is worth noting that in the regime $q<2$ even in the finite dimensional case it is a fairly recent result to derive well-posedness via PDE methods, see \cite{HuWeiYuan}.
Second, there is considerable interest in stochastic evolution equations driven by infinite-dimensional fractional Brownian motions, see e.g. \cite{DuncanMaslowski,MaslowskiNualart,TindelTudorViens, HuNualart,CoupekMaslowski}.
It is natural to pursue the possible regularization by noise effects in that context. A powerful feature of stochastic sewing is that it can extend to such non-Markovian settings with ease, while treating fractional noise with Kolmogorov equations seems out of reach even in the finite dimensional case.

Let us also mention briefly that, in the special case where the drift $b$ is of Nemitskii type, that is,
$b(x)(z):=f(x(z))$, for $f\in\C^\alpha(\R^d,\R)$, $x\in L^2([0,1]^d)$, $z\in[0,1]^d$, regularization by noise for the stochastic heat equation \eqref{eq:main-app} is well understood, and one can consider much less regular drifts than in the general case; $f$ can even be a Schwartz distribution. In particular, for $d=1$, strong existence and uniqueness are known for $\alpha > -1$ \cite{ABLM,D24}, while weak existence and uniqueness are known for $\alpha > -3/2$ \cite{BM24}.

\paragraph{The work \cite{AddonaBignamini2025}.} The first version of this paper was submitted to arXiv on 31~December 2025 and was restricted to the regime $\gamma\in[0,1)$, in particular it could handle the example \eqref{eq:main-app} in dimensions $d=1,2,3$. Shortly before its conclusion, on 22 December 2025, the work \cite{AddonaBignamini2025} by Addona-Bignamini-Orrieri-Scarpa appeared on arXiv independently. Therein, the condition \eqref{daprf} is also removed, and the almost identical condition $\gamma\in[0,1]$ is posed (cf. (H2) therein imposing  $\gamma/2\equiv\delta\leqslant1/2$), leading to the same scope in terms of dimensions for \eqref{eq:main-app}, but somewhat weaker conditions on $\alpha$. 
In the present version of our paper (which was submitted to arXiv on 13 February 2026), the range of exponents is significantly improved, allowing a nontrivial range of $\alpha$ for all $\gamma\in[0,3)$.
When it comes to the example \eqref{eq:main-app}, 
our result covers $1\leqslant d\leqslant 7$. For fixed $\gamma$ (or in the context of \eqref{eq:main-app}, fixed $d$), the condition on $\alpha$ is also improved: for example, in the case of $d=3$, \cref{cor:maincorollary} requires the more relaxed condition $\alpha>1/3$ on the H\"older exponent as opposed to $\alpha>2/3$ in \cite{AddonaBignamini2025}.

\paragraph{Convention on constants.} Throughout the paper, $C$ denotes a positive constant whose value may change from line to line; its dependence is always specified in the corresponding statement.

\paragraph{Acknowledgments.}
LA gratefully acknowledges that this research was funded in whole or in part by the Austrian Science Fund (FWF) [10.55776/STA119]. For open access purposes, the author has applied a CC BY public copyright license to any author-accepted manuscript version arising from this submission.
OB acknowledges funding by the Deutsche Forschungsgemeinschaft (DFG, German Research Foundation) --- CRC TRR SFB 388 ``Rough Analysis, Stochastic Dynamics and Related Fields'' - Project ID 516748464, subproject B08.  This paper is based upon work supported by the National Science Foundation under Grant No. DMS-2424139, while OB was in
residence at the Simons Laufer Mathematical Sciences Institute in Berkeley, California, during the Fall 2025 semester. LA, OB, and AS are very grateful to SLMath for creating a very good working atmosphere and providing excellent working conditions.
 MG was funded by the European Union (ERC, SPDE, 101117125). Views and opinions expressed
are however those of the author(s) only and do not necessarily reflect those of the European Union
or the European Research Council Executive Agency. Neither the European Union nor the granting
authority can be held responsible for them.

\section{Main results}

Let us introduce some basic notation.  For two normed spaces $(E,\|\cdot\|_E)$ and $(E',\|\cdot\|_{E'})$, a function $f\colon E \rightarrow E'$ and  $\alpha \in (0,1]$, we define the supremum and H\"older norms of $f$ as usual by 
\begin{align*}
\|f\|_\infty\coloneqq \sup_{x\in E} \|f(x)\|_{E'};\quad 
         [f]_{\cC^\alpha}\coloneqq \sup_{x \neq y} \frac{\|f(x)-f(y)\|_{E'}}{\|x-y\|_E^\alpha};\quad  \|f\|_{\cC^\alpha}\coloneqq \|f\|_\infty+[f]_{\cC^\alpha}.
\end{align*}
We say $f \in \cC^\alpha(E,E')$ if $\|f\|_{\cC^\alpha}<\infty$. For $m \in [1,\infty]$ and a probability space $(\Omega,\mathcal{F},\mathbb{P})$, the norm on $L^m(\Omega)$ is denoted by $\|\cdot\|_{L^m(\Omega)}$.

We fix a separable Hilbert space $(H,\langle\cdot,\cdot\rangle_H)$. First, we state our main assumptions on the operator $A$ and the parameter $\gamma$. 
\begin{assumption} \label{ass:A}
$A:D(A)\subset H\to H$ is a self-adjoint negative definite operator with purely atomic spectrum.
\end{assumption}
It follows from \cref{ass:A} that there exists  an orthonormal basis of $H$ consisting of eigenvectors of $-A$. We denote it by $(e_k)_{k\in\N}$ and let $(\lambda_k)_{k\in\N}$ be the corresponding sequence of positive eigenvalues. That is, one has
\begin{align} \label{eq:OpEV}
\Op e_k=-\lambda_k e_k,\quad k\in\N.
\end{align}
We suppose that the parameter $\gamma\ge0$ satisfies the following condition.
\begin{assumption} \label{ass:lambdagamma}
Assume
\begin{equation}\label{eq:gamma-condition}
        \sum_{k\in\N} \lambda_k^{-1-\gamma}<\infty.
    \end{equation}
\end{assumption}

Introduce now the stochastic convolution
\begin{equ}\label{eq:Z}
    Z_t:=\int_0^t e^{(t-r)\Op} (-\Op)^{-\gamma/2}dW_r,\quad t\geqslant0.
\end{equ}
Writing $W_t=\sum_{k\in\N} e_k\beta^k_t$, where $\beta^k$ is a sequence of independent one-dimensional Brownian motions, we see that
\begin{equ}\label{eq:mild-solution}
    \langle Z_t,e_k\rangle_H=\int_0^t e^{-(t-s)\lambda_k} \lambda_k^{-\gamma/2} d\beta^k_s \sim \mathcal{N}(0,\frac12\lambda_k^{-1-\gamma} (1-e^{-2t\lambda_k})).
\end{equ}
Therefore, $Z_t$ is an $H$-valued Gaussian random variable with mean $0$ and covariance operator 
\begin{equ}\label{eq:covariance-Z}
\mathcal{Q}_t=\frac{1}{2}(-\Op)^{-1-\gamma}(\mathrm{Id} - e^{2t\Op}),
\end{equ}
which by  \cref{ass:lambdagamma} is trace class on $H$. Moreover, $t\mapsto Z_t$ is continuous in $H$ almost surely.

Now we are ready to define the notion of a solution to our SDE. Without loss of generality, we consider the equation on the time horizon $[0,1]$.
\begin{definition}\label{def:sol}
Let $X\colon\Omega\times[0,1]\to H$ be a jointly measurable continuous process. We say that  $X$ a (mild) solution to SDE \eqref{eq:mainSDE} with the initial condition $x_0\in H$, if almost surely for all $t\in[0,1]$
\begin{equation*}
X_t= e^{t\Op}x_0+\int_0^t e^{(t-r)\Op}b(X_r) dr + Z_t.
\end{equation*}
\end{definition}

As usual, we define a \textit{weak solution} to \eqref{eq:mainSDE} as a couple $(X,W)$ on a complete filtered probability space $(\Omega, \F, (\F_t)_{t\in[0,1]},\PP)$ such that $W$ is an $( (\F_t)_{t\in[0,1]})$-cylindrical Brownian motion, $X$ is adapted to $(\F_t)$ and solves \eqref{eq:mainSDE} in the sense  of \cref{def:sol}. A weak solution $(X,W)$ is called a \textit{strong solution}  if $X$ is adapted to the augmented filtration generated by $W$. We say that  \textit{pathwise (strong) uniqueness} holds for \eqref{eq:mainSDE} if for any two weak solutions of \eqref{eq:mainSDE}  $(X,W)$ and $(\overline X,W)$  defined on the same filtered probability space one has $\PP(X_t=\overline X_t \text{ for all $t\in[0,1]$})=1$. 

The main result of our paper is strong existence and uniqueness of solutions to \eqref{eq:mainSDE}. The proof is presented in \cref{subsec:mainproof}.
\begin{theorem} \label{thm:main}
Let $x_0\in H$, $\alpha \in (0,1]$, $\gamma\geqslant 0$, $b \in \mathcal{C}^\alpha(H,H)$. Assume that \cref{ass:A,ass:lambdagamma} are satisfied and suppose further that 
\begin{equation} \label{eq:gammaupperbound}
        \alpha>\begin{cases}
            \dfrac{\gamma}{\gamma+1}&\text{if }0\leqslant\gamma\leqslant 1\vspace{0.2cm}
            \\
            \dfrac{1}{2}&\text{if }1\leqslant\gamma\leqslant \dfrac{3}{2}\vspace{0.2cm}
            \\
            \sqrt{\dfrac{\gamma-1}{2}}&\text{if }\dfrac{3}{2}\leqslant\gamma<3.
        \end{cases}
\end{equation}
Then there exists a strong solution to the SDE \eqref{eq:mainSDE} and pathwise uniqueness holds.
\end{theorem}
\begin{figure}[ht]
\centering
\begin{tikzpicture}
\begin{axis}[
    axis lines=left,
    xmin=0, xmax=3.1,
    ymin=0, ymax=1.05,
    xlabel={$\gamma$},
    ylabel style={rotate=-90, anchor=south},
    ylabel={$\alpha$},
    samples=200,
    thick,
]

\addplot[name path=top, draw=none] {1.05};

\addplot[name path=c1, draw=none, domain=0:1] {x/(x+1)};
\addplot[name path=c2, draw=none, domain=1:1.5] {0.5};
\addplot[name path=c3, draw=none, domain=1.5:3] {sqrt((x-1)/2)};
\addplot[name path=c4, draw=none, domain=0:1] {1};
\addplot[name path=c5, draw=none, domain=1:1.5] {1};
\addplot[name path=c6, draw=none, domain=1.5:3] {1};
\addplot[
    fill=gray!30,
    fill opacity=0.6,
    draw=none
] fill between[of=c4 and c1];

\addplot[
    fill=gray!30,
    fill opacity=0.6,
    draw=none
] fill between[of=c5 and c2];

\addplot[
    fill=gray!30,
    fill opacity=0.6,
    draw=none
] fill between[of=c6 and c3];

\addplot[thick, domain=0:1] {x/(x+1)};
\addplot[thick, domain=1:1.5] {0.5};
\addplot[thick, domain=1.5:3] {sqrt((x-1)/2)};

\end{axis}
\end{tikzpicture}
\caption{Regime of $\alpha,\gamma$ for pathwise uniqueness}
\end{figure}
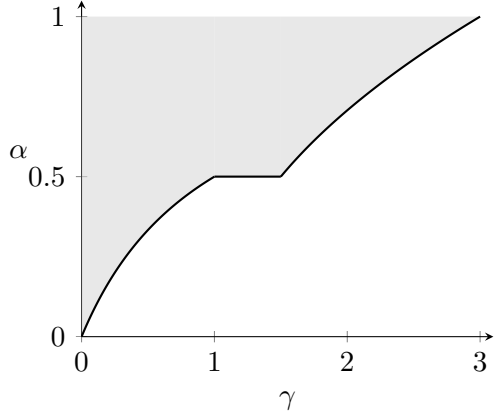

Concerning the main example \eqref{eq:main-app}, \cref{thm:main} implies the following.

\begin{corollary} \label{cor:maincorollary} 
Let $d\in\{1,2,\ldots,7\}$, $\alpha\in(0,1]$, $b\in\C^\alpha(L^2([0,1]^d),L^2([0,1]^d))$. 
Consider the  stochastic heat equation \eqref{eq:main-app} with the homogeneous Dirichlet boundary conditions.
Assume that $\alpha>\alpha^*$ given in the following table:
\[
\begin{array}{|c|c|c|c|c|c|c|c|}
\hline
d & 1 & 2 & 3 & 4 & 5 & 6 & 7 \\ \hline
\alpha^*
& 0
& 0
& \tfrac{1}{3}
& \tfrac{1}{2}
& \tfrac{1}{2}
& \tfrac{1}{\sqrt{2}}
& \tfrac{\sqrt{3}}{2}\\\hline
\end{array}
\]
Then there exists a $\gamma$ satisfying \cref{ass:lambdagamma} and \eqref{eq:gammaupperbound}, in which case \eqref{eq:main-app} has a unique strong solution.
\end{corollary}

\begin{remark}
If one equips \eqref{eq:main-app} with other boundary conditions, such as periodic, homogeneous Neumann or Robin, finitely many eigenvalues of the Laplacian can be nonnegative, violating \cref{ass:A}. The natural remedy is to let $A=\Delta-K\mathrm{id}$ with some large enough $K$ and replace $b$ by $\tilde b:=b+K\mathrm{id}$. Although $\tilde b$ does not immediately fall in the scope of \cref{thm:main} since it is not bounded, it is rather straightforward to modify the proof to accommodate such nonlinearities. We leave details to the interested reader.
\end{remark}

Note that if $d=1$, existence and uniqueness for the stochastic heat equation  \eqref{eq:main-app} is known for bounded measurable functions $b:L^2([0,1])\to L^2([0,1])$ for almost all initial conditions \cite{DFPR13,DFRV16}. This case remains beyond the scope of the methods developed in the present work.

Another interesting line of research is the extension to drifts $b$ that are H\"older continuous functions $E \to E$ for the Banach space $E := \C([0,1])$. For $d=1$, strong existence and uniqueness of \eqref{eq:main-app} with such drifts was shown in \cite{Cerrai13}. The case $d \ge 2$ remains open.

\begin{remark}
A perhaps unexpected feature of the result is that the regularization effect of $Z$ is not simply the time regularity of the process as an $H$-valued path. For example, in the context of \cref{cor:maincorollary}, it is easy to see that almost surely $Z\in \cap_{\eps>0}\cC^{\mu-\eps}([0,1],H)\setminus \cC^{\mu}([0,1],H)$ with $\mu=\gamma-(d/2-1)\vee 0$. 
A na\"ive analogy with finite dimensions would be the guess that $Z$ has the same regularizing effect as a fractional Brownian motion with Hurst parameter $\mu$, for which the well-known condition \cite{CG16} for strong well-posedness is 
\begin{equ}
    \alpha>1-\frac{1}{2\mu}.
\end{equ}
\end{remark}

\section{Proofs}

We split the main section of the paper into two parts. First, in \cref{subsec:averaging} we establish several results that capture the averaging effect of infinite-dimensional Gaussian measures. An important tool in this part is the Lasry-Lions approximation theory in Hilbert spaces. Then, using the stochastic sewing lemma in Hilbert spaces, we obtain a number of key integral bounds in \cref{subsec:mainproof}. The rest of the subsection is dedicated to the proof of \cref{thm:main}. To do so, we first prove existence and uniqueness to \eqref{eq:mainSDE} under assumptions that involve additionally introduced parameters (see \cref{a:manyineq}) to additionally measure regularity in Sobolev-scale. Then, the proof of \cref{thm:main} follows by showing that these assumptions coincide with the assumptions of \cref{thm:main}.

\subsection{Averaging properties of Gaussian measures} \label{subsec:averaging}

This subsection is structured as follows. We first give a brief summary of the definition of a Gaussian measure on a Hilbert space and the construction of its Cameron-Martin space. For further details, see \cite{Bogachev, Hairer}. We then state a general result on the averaging of a Gaussian measure $\mu$ on a Hilbert space against a bounded function $f$, yielding a Lipschitz estimate in directions of the associated Cameron-Martin space. This result, together with an interpolation argument (\cref{lem:interpolationShapo}), then allows us to quantify the averaging effect for a H\"older continuous function $f$ in \cref{lem:interpolation} and \cref{cor:Holderreg}.

Let $H$ be a separable Hilbert space and assume that the operator $A:D(A)\subset H\to H$ satisfies \cref{ass:A}. For $x\in H$ with $x=\sum_{k=1}^\infty x_k e_k$, and $\rho\ge0$, denote
\begin{equation*}
 \|x\|_{H^\rho}^2=\|(-A)^{\frac\rho2}x\|^2_{H}= \sum_{k=1}^\infty \lambda_k^\rho x_k^2
\end{equation*}
and let $H^\rho\subset H$ be the corresponding Sobolev space: 
\begin{equation*}
H^\rho:=\{x\in H:\,\, \|x\|_{H^\rho}<\infty\}.
\end{equation*}

Let $\mu$ be a centered   Gaussian measure $\mu$ on $H$, that is, $\mu$ is a Borel measure on $H$ and the pushforward of $\mu$ by any continuous linear functional is a centered Gaussian measure on $\mathbb{R}$. 
    Define the covariance operator $C_\mu\colon H\times H \rightarrow \mathbb{R}$ of $\mu$ by
\begin{align*}
    (h,k)\mapsto C_\mu(h,k)\coloneqq\int_{H} \langle h,x\rangle_H \langle k,x\rangle_H\, \mu(dx).
\end{align*}

The Cameron-Martin space $\H_\mu$ of the measure $\mu$ is the completion of the linear subspace
    \begin{equation*}
        \{h \in H:\exists h^*\in H \,\,\text{such that } C_\mu(h^*,l)=\langle l,h\rangle_H \,\,\, \forall \, l \in H\}
    \end{equation*}
under the norm $\|\cdot\|_{\CM}$ induced by the inner product $\langle h,k\rangle_{\CM}:= C_\mu(h^*,k^*)$, which makes it into a Hilbert space.
That this quantity is well defined follows from the fact that given $h\in \H_\mu$, any two $h_1^*$ and $h_2^*$ with the above property satisfies $\langle h_1^*-h_2^*,x\rangle_H=0$ for $\mu$-almost every $x\in H$. It is easy to see from the definition, that if $C_\mu$ is diagonal, that is, 
\begin{equation*}
C_\mu(e_i,e_j)=\sigma_i^2 \I_{i=j},\quad i,j\in \N,
\end{equation*}
where each $\sigma_i>0$, $i\in\N$, then for $x=\sum_{i=1}^\infty x_ie_i\in H$ we have
\begin{equation}\label{hmunorm}
\|x\|_{\CM}^2=\sum_{i=1}^\infty \sigma_i^{-2}x_i^2.
\end{equation}
It is also immediate to see, that if a random variable $X$ has law $\mu$, then the  random variable $\langle h^*, X\rangle_H$ is Gaussian with mean $0$ and variance $\|h\|^2_{\CM}$.

For $h\in \H_\mu$, define the map $T_h: H\to H$ by $T_h(x):=x+h$. Then, by the Cameron-Martin theorem, the measure $T_h^\# \mu$ is absolutely continuous with respect to $\mu$ and the Radon-Nikodym derivative is given by 
\begin{equation}\label{eq:R-N}
\frac{d T_h^\# \mu}{d \mu}(x)=\exp\Big\{\langle h^*,x\rangle_H-\frac{1}{2}\|h\|_{\CM}^2\Big\},\quad x\in H.
\end{equation}

We begin with the first averaging lemma.
\begin{lemma} \label{lem:averagingproperties}
There exists a constant $C>0$ such that for any bounded measurable function $f\colon H\to H$ and any $h_1,h_2,h_3 \in \mathcal{H}_{\mu}$, one has
\begin{align}
&\Big\|\int_H f(x+h_1)-f(x+h_2) \mu(dx)\Big\|_{H}\leqslant C\|f\|_\infty \|h_1-h_2\|_{\CM} \label{eq:2point},\\
&\Big\|\int_H f(x+h_1)-f(x+h_2)-f(x+h_3)+f(x+h_3+h_2-h_1) \mu(dx)\Big\|_{H} \nonumber\\
    &\qquad\le C \|f\|_\infty \|h_3-h_1\|_{\CM} \|h_2-h_1\|_{\CM}. \label{eq:4point}
\end{align}
\end{lemma}
\begin{proof}
First, let us prove  \eqref{eq:2point}.
It suffices to consider the case $h_2=0$; the case $h_2\neq 0$ then easily follows by shifting $f$. Note that if $\|h_1\|_{\CM}\geqslant 1$, then \eqref{eq:2point} holds trivially. Henceforth we assume $\|h_1\|_{\CM}\le1$. Denote $\zeta_h(x):=\langle h^*,x\rangle_H-\frac{1}{2}\|h\|_{\CM}^2$. 
Using \eqref{eq:R-N} and the inequality
\begin{equation}\label{eq:very-triv}
     |e^{a}-1|\le (1\vee e^{a})|a|
\end{equation}
valid for any $a\in \R$, we derive
\begin{align}\label{2pointmany}
\left\|\int_H f(x+h_1)-f(x)\mu(dx)\right\|_{H}&=\left\|\int_H f(x)  T_{h_1}^\#\mu (dx)-\int_H f(x)\mu(dx)\right\|_{H}\nn\\
&=\left\|\int_H f(x)\Bigl(\frac{d T_{h_1}^\# \mu}{d \mu}(x)-1\Bigr) \mu(dx)\right\|_{H}\nn\\
&\leqslant \|f\|_\infty \int_H (1\vee e^{\zeta_{h_1}(x)}) |\zeta_{h_1}(x)|\mu(dx)\nn\\
&\leqslant \|f\|_\infty \|1\vee e^{\zeta_{h_1}(\cdot)}\|_{L^2(H,\mu)} \|\zeta_{h_1}(\cdot)\|_{L^2(H,\mu)}.
\end{align}

Recalling the variance of $\zeta_{h_1}$, we see that 
\begin{equation}\label{eq:constant1}
\|1\vee e^{\zeta_{h_1}(\cdot)}\|_{L^2(H,\mu)}^2
\le 1+\exp(\|h_1\|^2_{\CM})\le 4,
\end{equation}
since we consider the case $\|h_1\|_{\CM}\le1$. 
Again using the same arguments and the inequality $(a+b)^2\le 2 a^2 +2 b^2$ valid for any $a,b\in\R$ we get,
\begin{equation}\label{eq:constant2}
\|\zeta_{h_1}(\cdot)\|_{L^2(H,\mu)}^2\le 2 
 \|\langle h_1^*,\cdot\rangle \|_{L^2(H,\mu)}^2 +\frac12\|h_1\|_{\CM}^4 \le 3\|h_1\|_{\CM}^2.
\end{equation}
Combining this with \eqref{eq:constant1} and substituting into \eqref{2pointmany}, we get  \eqref{eq:2point}.

Now we move on to the proof of  \eqref{eq:4point}. We note again that it suffices to prove this inequality  for $h_1=0$ as the case $h_1\neq 0$ then easily follows by shifting $f$. If $\|h_2\|_{\CM}\ge1$ and $|h_3\|_{\CM}\ge 1$, \eqref{eq:4point} is immediate.
It remains to treat the cases $\|h_3\|_{\CM}\leqslant 1 \leqslant \|h_2\|_{\CM}$ (the case of both inequalities reversed follows by symmetry) and $\|h_3\|_{\CM}\le 1$, $\|h_2\|_{\CM} \leqslant 1$.

First, consider the case $\|h_3\|_{\CM}\leqslant 1 \leqslant \|h_2\|_{\CM}$.
Using \eqref{eq:2point} for the functions $f$ and $T_{h_2}f$, we get
\begin{align*}
    \Big\|\int_H &f(x)-f(x+h_2)-f(x+h_3)+f(x+h_3+h_2) \mu(dx)\Big\|_{H} \\ &=\Big\|\int_H f(x)-f(x+h_3)-(T_{h_2}f(x)-T_{h_2}f(x+h_3)) \mu(dx) \Big\|_{H}\\
    &\leqslant C(\|f\|_\infty + \|T_{h_2}f\|_\infty)\|h_3\|_{\CM}\le C \|f\|_\infty \|h_2\|_{\CM}\|h_3\|_{\CM},
\end{align*}
with a universal constant $C>0$. This yields \eqref{eq:4point}.

Next, consider the case $\|h_2\|_{\CM}\le1$, $\|h_3\|_{\CM} \leqslant 1$. 
We have
\begin{align}\label{bigstuff}
&\Big\|\int_H f(x)-f(x+h_2)-f(x+h_3)+f(x+h_2+h_3) \mu(dx)\Big\|_{H}\nn\\
&\qquad=\Big\|\int_H f(x)\mu(dx) -\int_H f(x)T_{h_2}^\#\mu(dx) -\int_H f(x)T_{h_3}^\#\mu(dx) +\int_H f(x)T_{h_2+h_3}^\#\mu(dx) \Big\|_{H}\nn\\
&\qquad= \Big\|\int_H f(x)(1-e^{\zeta_{h_2}(x)}-e^{\zeta_{h_3}(x)}+e^{\zeta_{h_2+h_3}(x)})\mu(dx)\Big\|_{H}\nn\\
&\qquad\leqslant \Big\|\int_H f(x)(e^{\zeta_{h_2}(x)}-1)(e^{\zeta_{h_3}(x)}-1)\mu(dx)\Big\|_{H}\nn\\
&\qquad{\phantom{\le}} +\Big\|\int_H f(x)\Big(e^{\zeta_{h_2+h_3}(x)}-e^{\zeta_{h_2}(x)}e^{\zeta_{h_3}(x)}\Big) \mu(dx)\Big\|_{H}=:I_1+I_2.
\end{align}
The first summand $I_1$ is treated similarly to the proof of \eqref{eq:2point}. Using \eqref{eq:very-triv} and H\"older's inequality, we get 
\begin{align*}
I_1&\le \|f\|_\infty\Big\|\int_H  (1\vee e^{\zeta_{h_2}(x)}) |\zeta_{h_2}(x)|(1\vee e^{\zeta_{h_3}(x)}) |\zeta_{h_3}(x)| \mu(dx)\Big\|_{H}\\
&\le\|f\|_\infty \prod_{i=2,3} \|1\vee e^{\zeta_{h_i}(\cdot)}\|_{L^4(H,\mu)} \|\zeta_{h_i}(\cdot)\|_{L^4(H,\mu)}.
\end{align*}
Similarly to \eqref{eq:constant1} and \eqref{eq:constant2}, we derive for $i=2,3$
\begin{equation*}
\|1\vee e^{\zeta_{h_i}(\cdot)}\|_{L^4(H,\mu)}^4\le C;\qquad
\|\zeta_{h_i}(\cdot)\|_{L^4(H,\mu)}^4\le C \|h_i\|_{\CM}^4   
\end{equation*}
for a universal constant $C>0$. We can conclude
\begin{equation}\label{ivangoal}
    I_1\leqslant C\|f\|_\infty\|h_2\|_{\CM}\|h_3\|_{\CM}. 
\end{equation}
As for $I_2$, we have $e^{\zeta_{h_2+h_3}(x)}-e^{\zeta_{h_2}(x)}e^{\zeta_{h_3}(x)}=e^{\zeta_{h_2+h_3}(x)}(1-e^{\langle h_2,h_3\rangle_{\CM}})$. Therefore
\begin{align*}
I_2&=|1-e^{\langle h_2,h_3\rangle_{\CM}}|\,\Big\|\int_H f(x)T_{h_2+h_3}^\# \mu(dx)\Big\|_{H}\\
&\le  e\|f\|_\infty|\langle h_2,h_3\rangle_{\CM}|\leqslant e \|f\|_\infty \|h_2\|_{\CM} \|h_3\|_{\CM}.
\end{align*}
Combining this with \eqref{ivangoal} and \eqref{bigstuff}, we get 
\eqref{eq:4point}.
\end{proof}

Our next goal is to prove an analogue of \cref{lem:averagingproperties}, but for $f$ H\"older continuous. This is done in \cref{lem:interpolation} using the Lasry-Lions regularization method, proposed in \cite[Remark~(iv)]{LasryLions}. For the sake of completeness we also provide a proof thereof.
\begin{lemma} \label{lem:interpolationShapo}
Let $(E,\|\cdot\|_E)$ be a Banach space, $\lambda>0$, $\alpha\in(0,1)$, and $g \in \cC^\alpha(E,\mathbb{R})$. Define
\begin{align}\label{eq:magic}
    g_\lambda(x)\coloneqq \inf_{t \in E}\Big[g(x+t)+\lambda^{-1}\|t\|_E\Big].
\end{align}
Then for any $x,y\in E$, one has 
\begin{align}
|g_\lambda(x)-g_\lambda(y)|&\leqslant \lambda^{-1} \|x-y\|_E,\label{eq:flambdaLip}\\
|g(x)-g_\lambda(x)|&\leqslant [g]_{\cC^\alpha}^{\frac{1}{1-\alpha}} \lambda^{\frac{\alpha}{1-\alpha}}.\label{eq:flambdaf}
\end{align}
\end{lemma}

\begin{proof}
We start with \eqref{eq:flambdaLip}. Fix $\lambda>0$. Note that for any $x\in E$, by continuity, the infimum in \eqref{eq:magic} is attained at some point $t^*_x\in E$.
Note also that \eqref{eq:flambdaLip} is symmetric in $x$ and $y$ so it is no loss of generality to assume $g_\lambda(x)\geq g_\lambda(y)$ and the absolute value on the left-hand side can be dropped.
Then we have that
\begin{align*}
    g_\lambda(x)&\leqslant g(x+y-x+t^*_y)+\lambda^{-1}\|y-x+t^*_y\|_E\\
    &\leqslant g(y+t^*_y)+\lambda^{-1}\|t^*_y\|_E+\lambda^{-1}\|x-y\|_E\\
    &=g_\lambda(y) +\lambda^{-1} \|x-y\|_E,
\end{align*}
as claimed.

Moving on to \eqref{eq:flambdaf}, since $g(x)\geqslant g_\lambda(x)$ by definition, the absolute value on the left-hand side can again be dropped.
Note that
\begin{equation*}
    g_\lambda(x)= g(x+t^*_x)+\lambda^{-1}\|t^*_x\|
    \geqslant g(x)+\lambda^{-1}\|t^*_x\|-[g]_{\cC^\alpha}\|t^*_x\|^\alpha.
\end{equation*}
Therefore, for any $x\in E$ we have
\begin{equation*}
g(x)-g_\lambda(x)\leqslant \sup_{T\geq 0}\big([g]_{\cC^\alpha}T^\alpha-\lambda^{-1} T\big).
\end{equation*}
An elementary computation gives that the supremum is attained at $T=\big(\alpha [g]_{\cC^\alpha}\lambda\big)^{\frac{1}{1-\alpha}}$. Hence
\begin{equation*}
g(x)-g_\lambda(x)\le (1-\alpha)\,\alpha^{\frac{\alpha}{1-\alpha}} [g]_{\mathcal{C}^\alpha}^{\frac{1}{1-\alpha}}\lambda^{\frac{\alpha}{1-\alpha}},
\end{equation*}
which gives the desired bound \eqref{eq:flambdaf}.
\end{proof}

Now we are ready to extend \cref{lem:averagingproperties} to H\"older continuous functions. 
\begin{lemma} \label{lem:interpolation}
There exists $C>0$ such that for any $\alpha \in (0,1]$, $f\in \cC^\alpha(H,H)$ and any $h_1,h_2,h_3 \in \mathcal{H}_{\mu}$, one has 
\begin{align}
&\Big\|\int_H (f(x+h_1)-f(x+h_2))\, \mu(dx)\Big\|_H\le  C[f]_{\C^\alpha} \|h_1-h_2\|_H^\alpha \|h_1-h_2\|_{\CM}^{1-\alpha},\label{eq:twopointHolder}\\
&\Big\|\int_H (f(x+h_1)-f(x+ h_2)-f(x+h_3) +f(x+h_2 +h_3-h_1))\, \mu(dx)\Big\|_H\nonumber\\
 &\qquad\le C[f]_\alpha\|h_1-h_2\|_{\CM} (\|h_1-h_3\|_{\CM}^{1-\alpha}\wedge1)\|h_1-h_3\|_H^\alpha.\label{eq:fourpointHolder}
\end{align}
\end{lemma}

\begin{proof}
We assume $h_1-h_2\neq 0$ and $[f]_{\C^\alpha}\neq 0$ to avoid the trivial cases. As in \cref{lem:averagingproperties}, without loss of generality, we assume that $h_2=0$.
For $u \in H$, let $g^u(x)\coloneqq \langle u ,f(x)\rangle_H$, $x\in H$, and for $\lambda>0$ define $g_\lambda^u(x)$ as in \cref{lem:interpolationShapo}. 

First, let us prove \eqref{eq:twopointHolder}. We derive
\begin{align}\label{igensh}
\Big\|\int_H (f(x+{h_1})-f(x)) \mu(dx)\Big\|_H&=\sup_{u\in H:\, \|u\|_H\leqslant 1} \Big\langle u,\int_H (f(x+{h_1})-f(x))\, \mu(dx)\Big\rangle_H\nn\\
    &=\sup_{u\in H:\, \|u\|_H\leqslant 1} \int_H (g^u(x+{h_1})-g^u(x))\, \mu(dx)\nn\\
    &\leqslant \sup_{u\in H:\, \|u\|_H\leqslant 1} \int_H (g_\lambda^u(x+{h_1})-g_\lambda^u(x))\, \mu(dx)\nn\\
    &\phantom{\le}+\sup_{u\in H:\, \|u\|_H\leqslant 1} \int_H ((g^u-g^u_\lambda)(x+{h_1})-(g^u-g^u_\lambda)(x))\, \mu(dx)\nn\\
    &=:I_1+I_2.
\end{align}
It follows from \eqref{eq:flambdaLip}, that
\begin{equation}\label{ionesh}
I_1\le \lambda^{-1}\|h_1\|_H.
\end{equation}
Applying \eqref{eq:2point}, we get
\begin{equation}\label{itwosh}
I_2\le C \|{h_1}\|_{\CM}\sup_{u\in H:\, \|u\|_H\leqslant 1} \|g^u-g^u_\lambda\|_{\infty}\le  C \|{h_1}\|_{\CM}\lambda^{\frac{\alpha}{1-\alpha}} \sup_{u\in H:\, \|u\|_H\leqslant 1}[g^u]_{\cC^\alpha}^{\frac{1}{1-\alpha}} ,
\end{equation}
with a universal constant $C>0$, where in the second inequality we used again \cref{lem:interpolationShapo}.
Using   the definition of $g^u$, we immediately get  
\begin{equation}\label{gfh}
[g^u]_{\cC^\alpha}\leqslant [f]_{\cC^\alpha},\text{\quad for any $\|u\|_H\leqslant 1$}.   
\end{equation}
Hence, substituting this into \eqref{itwosh} and combining with \eqref{igensh}, \eqref{ionesh}, we finally get
\begin{equation*}
\Big\|\int_H (f(x+{h_1})-f(x)) \mu(dx)\Big\|_H\le 
 \lambda^{-1}\|h_1\|_H+C \lambda^{\frac{\alpha}{1-\alpha}}\|{h_1}\|_{\CM} [f]_{\cC^\alpha}^{\frac{1}{1-\alpha}} 
\end{equation*}
with a universal constant $C>0$. By choosing now $\lambda=\Big(\frac{\|{h_1}\|_H}{\|{h_1}\|_{\CM}}\Big)^{1-\alpha}[f]_{\cC^\alpha}^{-1}$, \eqref{eq:twopointHolder}  follows.

The proof of \eqref{eq:fourpointHolder} is very similar to the proof of \eqref{eq:twopointHolder}. As usual, without loss of generality assume  $h_1=0$, ${h_3\neq 0}$, $[f]_{C^\alpha}\neq0$. Define $g^{u,h}_\lambda$ analogously to $g^u_\lambda$ in the first part of the proof, but starting from $T_hf(x)=f(x+h)$ instead of $f$. Arguing as in \eqref{igensh}, we write
\begin{align*}
&\Big\|\int_H f(x)-f(x+h_2)-f(x+h_3)+f(x+h_2+h_3)\,\mu(dx)\Big\|_H\\
&\quad=\sup_{\|u\|_H\leqslant 1} \int_H (g^u(x)-g^u(x+h_2)-g^u(x+h_3)+g^u(x+h_2+h_3)\, \mu(dx)\\
&\quad\leqslant \sup_{\|u\|_H\leqslant 1} \int_H (g_\lambda^u-g_\lambda^{u,h_3})(x)-(g_\lambda^u-g_\lambda^{u,h_3})(x+h_2)\,\mu(dx)\\
&\quad\phantom{\le} +\sup_{\|u\|_H\leqslant 1} \int_H \bigl((g^u-g^u_\lambda)(x)-(g^u-g^u_\lambda)(x+h_2)
\\
&\qquad\qquad\qquad-(g^u-g^u_\lambda)(x+h_3)-(g^u-g^u_\lambda)(x+h_2+h_3)\bigr)\, \mu(dx)\\
&\quad \le C\sup_{\|u\|_H\leqslant 1}\Big(\|h_2\|_{\CM} \|g^u_\lambda-g^{u,h_3}_\lambda\|_\infty + \|h_2\|_{\CM} (\|h_3\|_{\CM}\wedge1) \|g^u-g^u_\lambda\|_\infty\Big)\\
&\quad \le C \sup_{\|u\|_H\leqslant 1}\Big(\lambda^{-1}\|h_2\|_{\CM} \|h_3\|_H+\|h_2\|_{\CM} (\|h_3\|_{\CM}\wedge1)[g^u]_{\cC^\alpha}^{\frac{1}{1-\alpha}}\lambda^{\frac{\alpha}{1-\alpha}}\Big)\\
&\quad \le C\|h_2\|_{\CM} \Big(\|h_3\|_H\lambda^{-1}+ (\|h_3\|_{\CM} \wedge1)[f]_{\cC^\alpha}^{\frac{1}{1-\alpha}}\lambda^{\frac{\alpha}{1-\alpha}}\Big),
\end{align*}
with a universal constant $C>0$. Here we used both parts of \cref{lem:averagingproperties} in the second inequality, \cref{lem:interpolationShapo} in the third inequality and \eqref{gfh} in the last inequality. 
By choosing $\lambda=\Big(\frac{\|{h_3}\|_H}{\|{h_3}\|_{\CM}\wedge1}\Big)^{1-\alpha}[f]_{\cC^\alpha}^{-1}$, inequality \eqref{eq:fourpointHolder} follows.
%
\end{proof}

Next, let us understand how the operators $A$,  $e^{t A}$ and the norms $\|\cdot\|_{H}$, $\|\cdot\|_{H^\rho}$, $\|\cdot\|_{\CM}$ interact with each other. The first result is well-known; we provide a short proof for the sake of completeness.
\begin{proposition}\label{p:basic}
Let $\rho>0$. Then there exists a constant $C=C(\rho)>0$ such that for any $t>0$, $x\in H$ we have
\begin{equation}\label{mainrhoh}
\|e^{t A} x\|_{H^\rho}\le C t^{-\frac\rho2}\|x\|_H.
\end{equation}
\end{proposition}
\begin{proof}
Let $x=\sum_{k=1}^\infty x_k e_k\in H$. Then it follows by definition that
\begin{equation*}
\|e^{t A} x\|^2_{H^\rho}=\sum_{k=1}^\infty e^{-2 t \lambda_k}\lambda_k^\rho x_k^2= t^{-\rho}\sum_{k=1}^\infty e^{-2 t \lambda_k}(t\lambda_k)^\rho x_k^2\le t^{-\rho}\|x\|^2_H \sup_{z\ge0} e^{-2z}z^\rho\le C t^{-\rho}\|x\|^2_H,
\end{equation*}
which implies \eqref{mainrhoh}.
\end{proof}

The second result shows how that if the covariance operator of a Gaussian measure satisfies certain assumptions, then one can switch from the norm  $\|\cdot\|_{\CM}$ to the norm $\|\cdot\|_{H^\rho}$ for a certain ``price''. From now on, we additionally impose in this subsection that   the operator $A$ also satisfies \cref{ass:lambdagamma}.

Let $\mu_t:=\law(Z_t)$, where $Z$ is the stochastic convolution introduced in \eqref{eq:Z} and $t>0$.  Denote by $\H_t$ the Cameron-Martin space of $\mu_t$. 
\begin{lemma}\label{l:cmbound}
Let $\rho\in[0,\gamma]$.
Then there exists a constant $C=C(\gamma)>0$ such that for any $t>0$, $x\in H$ we have
\begin{equation}\label{maincm}
\|e^{t A} x\|_{\H_{t}}\le C t^{-\frac12(1+\gamma-\rho)}\|x\|_{H^\rho}.
\end{equation}
\end{lemma}

\begin{proof}
It follows from \eqref{eq:covariance-Z}, that for any $i,j\in\N$ we have
\begin{equation*}
C_{\mu_t}(e_i,e_j)=\int_H x_i x_j \mu_t (dx)=\frac12\lambda_i^{-1-\gamma}(1-e^{-2t \lambda_i})\I_{i=j}.
\end{equation*}
Take now any $x=\sum_{k=1}^\infty x_k e_k\in H$. 
Then, recalling \eqref{hmunorm}, we derive
\begin{align*}
\|e^{t A} x\|^2_{\H_t}&=2\sum_{k=1}^\infty e^{-2 t \lambda_k}\lambda_k^{1+\gamma}(1-e^{-2t \lambda_k})^{-1} x_k^2\\
&= 2 t^{-(1+\gamma-\rho)}\sum_{k=1}^\infty e^{-2 t \lambda_k}(t \lambda_k)^{1+\gamma-\rho} (1-e^{-2t \lambda_k})^{-1} \lambda_k^{\rho} x_k^2\\
&\le 2  t^{-(1+\gamma-\rho)}\|x\|_{H^\rho}^2 \sup_{x\ge0} e^{-2 x}x^{1+\gamma-\rho} (1-e^{-2x })^{-1} \\
&\le C t^{-(1+\gamma-\rho)}\|x\|_{H^\rho}^2,
\end{align*}
for $C=C(\gamma)$. Here in the last line we took into account that $1+\gamma-\rho\ge1$ by assumption, and hence the supremum is finite. This yields \eqref{maincm}.
\end{proof}

Combining all the statements of this subsection, we obtain the following technical result, which will be crucial for verifying the conditions of the stochastic sewing lemma.

\begin{corollary}\label{cor:Holderreg}
Let $\alpha \in (0,1]$, $\gamma \geqslant 0$ and $\rho \in [0,\gamma]$. Then there exists a constant $C=C(\gamma,\rho)>0$ such that for any $f \in \cC^\alpha(H,H)$, $h_1,h_2,h_3\in H$ and $t \in (0,T]$,
    \begin{align} 
      &\Big\|\int_H f(x+e^{t\Op}h_1)-f(x+e^{t\Op}h_2)\mu_t(dx)\Big\|_{H} \leqslant C t^{\frac{-(1+\gamma-\rho)(1-\alpha)}{2}} [f]_{\mathcal{C}^\alpha}\|h_1-h_2\|_{H}^\alpha\|h_1-h_2\|_{H^\rho}^{1-\alpha};\label{eq:Holderregtwopoint}\\
        &\Big\|\int_H f(x+e^{t\Op}h_1)-f(x+e^{t\Op} h_2)-f(e^{t\Op} h_3) +f(x+e^{t\Op} h_2 +e^{t\Op}h_3-e^{t\Op}h_1) \mu_t(dx)\Big\|_H\nonumber\\
        &\qquad\qquad\leqslant C t^{-\frac12(1+\gamma)(2-\alpha)+\frac\rho2}[f]_{\mathcal{C}^\alpha}\|h_1-h_3\|_{H} \|h_1-h_2\|_{H^\rho}.      \label{eq:fourpoint}\\
        &\Big\|\int_H f(x+e^{t\Op}h_1)-f(x+e^{t\Op} h_2)-f(e^{t\Op} h_3) +f(x+e^{t\Op} h_2 +e^{t\Op}h_3-e^{t\Op}h_1) \mu_t(dx)\Big\|_H\nonumber\\
        &\qquad\qquad\leqslant C t^{-\frac12(1+\gamma)+\frac\rho2}[f]_{\mathcal{C}^\alpha}\|h_1-h_3\|^\alpha_{H} \|h_1-h_2\|_{H^\rho}.      \label{eq:modfourpoint}
    \end{align}
\end{corollary}

\begin{proof}

It follows from \eqref{eq:twopointHolder} of \cref{lem:interpolation} and \cref{l:cmbound} that
\begin{align*}
    \Big\|\int_H f(x+e^{t\Op}h_1)-f(x+e^{t\Op}h_2)\mu_t(dx)\Big\|_H&\le C [f]_{\C^\alpha}\|e^{t\Op}(h_1-h_2)\|_{H}^\alpha\|e^{t\Op}(h_1-h_2)\|_{\mathcal{H}_t}^{1-\alpha}\\
    &\le C t^{\frac{-(1+\gamma-\rho)(1-\alpha)}{2}} [f]_{\mathcal{C}^\alpha}\|h_1-h_2\|_{H}^\alpha\|h_1-h_2\|_{H^\rho}^{1-\alpha}.
\end{align*}
for $C=C(\gamma)>0$. This implies \eqref{eq:Holderregtwopoint}. 

Similarly, using \eqref{eq:fourpointHolder}, we get
\begin{align}\label{fourpc}
&\Big\|\int_H (f(x+h_1)-f(x+ h_2)-f(x+h_3) +f(x+h_2 +h_3-h_1))\, \mu_t(dx)\Big\|_H\nonumber\\
&\qquad\le C[f]_\alpha\|e^{tA}(h_1-h_2)\|_{\H_t} (\|e^{tA}(h_1-h_3)\|_{\H_t}^{1-\alpha}\wedge1)\|e^{tA}(h_1-h_3)\|_H^\alpha
\end{align}
Using \cref{p:basic,l:cmbound} we deduce 
\begin{align}
&\|e^{tA}(h_1-h_2)\|_{\H_t}\le C  t^{-\frac12(1+\gamma-\rho)}\|h_1-h_2\|_{H^\rho};\label{goodin1}\\
&\|e^{tA}(h_1-h_3)\|_{H}^\alpha\le C\|h_1-h_3\|_{H}^\alpha;\label{goodin2}\\
&\|e^{tA}(h_1-h_3)\|_{\H_t}^{1-\alpha}
\le t^{-\frac12(1+\gamma)(1-\alpha)}\|h_1-h_3\|^{1-\alpha}_{H}\nn
\end{align}
for $C=C(\gamma)>0$. Combining these inequalities with \eqref{fourpc},
we get~\eqref{eq:fourpoint}. Similarly, using the bound $\|e^{tA}(h_1-h_3)\|_{\H_t}^{1-\alpha}\wedge1\le 1$, \eqref{eq:modfourpoint} also follows from \eqref{fourpc}, \eqref{goodin1} and \eqref{goodin2}.
\end{proof}

\subsection{Integral estimates and proof the of main result}\label{subsec:mainproof}

After the results on the averaging properties of Gaussian measures in the previous subsection, we now proceed with the proof of the main result.

First, we recall our main technical tools: the stochastic sewing lemma for Hilbert space valued stochastic processes \cite[Theorem 3.1]{Le} and the taming singularities lemma \cite[Lemma~2.3]{BFG}. One feature of our proof strategy for strong uniqueness is that we do buckling \textit{simultaneously} in two spaces: $\C^\tau([0,1],H)$ and $\C^{\tau-\frac\rho2}([0,1],H^\rho)$ for certain values of $\tau,\rho>0$. Accordingly, we must capture the regularization effect of $Z$ perturbed by generic processes both in $H$ and in $H^\rho$.
This is achieved in \cref{lem:comparison,c:comparison} by applying a stochastic sewing argument in Hilbert spaces under \cref{a:manyineq}, which introduces additional parameters that allow us to fine tune between these spaces. We then apply this result to the specific perturbation given by the drift part of mild solutions to \eqref{eq:mainSDE} in order to prove existence and uniqueness. Finally, in \cref{lem:exponents}, we show that the assumptions in \cref{thm:main} coincide with \cref{a:manyineq}, and hence \cref{thm:main} follows.

For $0 \leqslant S < T$, consider the simplices
\begin{equation*}
	\Delta^2_{[S,T]}:=\{(s,t)\in[S,T]^2\colon s< t\};\quad \Delta^3_{[S,T]}:=\{(s,u,t)\in[S,T]^3\colon s< u< t\}.
\end{equation*}
For brevity, we sometimes write $\Delta_{[S,T]} := \Delta^2_{[S,T]}$.
For $A_{\cdot,\cdot}:\Delta^2_{[S,T]}\to H$, we denote
\begin{equation}\label{deltadiff}
	\delta A_{s,u,t}:= A_{s,t}-A_{s,u}-A_{u,t},\quad  (s,u,t)\in\Delta^3_{[S,T]}. 
\end{equation}
For a filtration $(\mathcal{F}_t)_{t\geqslant 0}$, we denote the conditional expectation $\EE[\cdot\vert \mathcal{F}_t]$ by $\EE^t[\cdot]$.

\begin{proposition}[Stochastic sewing in Hilbert spaces, {\cite[Theorem 3.1]{Le}}] \label{lem:SSL}
Let $H$ be a separable Hilbert space and $m \geqslant 2$. Let $A_{\cdot,\cdot}\colon \Omega \times \Delta^2_{[S,T]}\rightarrow H$ be measurable. Assume that there exist constants $\Gamma_1,\Gamma_2, \Gamma_3\geqslant 0$, $\varepsilon_1,\varepsilon_2,\eps_3>0$, such that, for all $(s,u,t)\in[S,T]^3$, $A_{s,t}$ is $\mathcal{F}_t$-measurable and the following bounds hold
\begin{align}
\big\|\|\delta A_{s,u,t}\|_H\big\|_{L^m(\Omega)}&\leqslant \Gamma_1|t-s|^{\frac12+\varepsilon_1}, \label{eq:SSL2}
\\
    \big\|\|\EE^s[\delta A_{s,u,t}]\|_H\big\|_{L^m(\Omega)}&\leqslant  \Gamma_2|t-s|^{1+\varepsilon_2}+\Gamma_3|t-s|^{1+\varepsilon_3}\label{eq:SSL1}.    
\end{align}
Then there exists a process $(\mathcal{A}_t)_{t\in [S,T]}$ such that, for any $t \in [S,T]$ 
and any sequence of partitions $\Pi_k=\{t_i^k\}_{i=0}^{N_k}$ of $[S,t]$ with mesh size going 
to zero, 
\begin{align} \label{eq:SSLconv}
    \mathcal{A}_t=\lim_{k\rightarrow \infty}\sum_{i=0}^{N_k}A_{t_i^k,t_{i+1}^k} \text{ in probability.} 
\end{align}
Furthermore, there exists a constant $C=C(\varepsilon_1,\varepsilon_2,\eps_3,m)$ independent of $S,T$ such that for every $(s,t) \in \Delta^2_{[S,T]}$,
\begin{equation}\label{sslmres}
    \big\|\|A_{s,t}-\mathcal{A}_t-\mathcal{A}_s\|_H\big\|_{L^m(\Omega)}\leqslant C\Gamma_1 (t-s)^{\frac12+\varepsilon_1}+C \Gamma_2 (t-s)^{1+\varepsilon_2}+C \Gamma_3 (t-s)^{1+\varepsilon_3}.
\end{equation}
\end{proposition}

\begin{proposition}[Taming singularities lemma, \cite{BFG,le2021taming}]\label{e:mbb2}
Let $(E,d)$ be a metric space, $T\in(0,1]$ and $h \in \mathbb{N}$. Suppose there exist constants $\tau_i,\eta_i \ge 0$ with $\tau_i>\eta_i$, $i=1,\dots,h$, and $\Gamma_i>0$, such that a function $F\colon(0,T]\to E$ satisfies
\begin{equation*}
d(F_s,F_t) \le \sum_{i=1}^h\Gamma_i s^{-\eta_i}(t-s)^{\tau_i}\quad\text{for any $0< s\le t\le T$}.
\end{equation*}
Then $F$ can be continuously extended at $0$ and there exist constants $C_i=C_i(\eta_i,\tau_i)>0$ such that
\begin{equation}\label{rezult}
d(F_s,F_t)\le \sum_{i=1}^h C_i \Gamma_i(t-s)^{\tau_i- \eta_i}\quad\text{for any $0\le s\le t\le T$.}
\end{equation}	
\end{proposition}

For a process $Y\colon \Omega\times [S,T] \rightarrow H$, $\tau \in (0,1]$, $\rho\ge0$, $m\in[1,\infty]$, we define the following H\"older-type seminorm:
\begin{align*}[Y]_{\mathcal{C}^\tau_{[S,T],H^\rho,L^m}}\coloneqq \sup_{(s,t) \in \Delta^2_{[S,T]}} \frac{\big\|\|Y_t-e^{(t-s)\Op}Y_s\|_{H^\rho}\big\|_{L^m(\Omega)}}{(t-s)^\tau}.
\end{align*}
We define the corresponding norm by
\begin{align*}
\|Y\|_{\mathcal{C}^\tau_{[S,T],H^\rho,L^m}}:=\big\|\|Y_S\|_{H^\rho}\big\|_{L^m(\Omega)}+[Y]_{\mathcal{C}^\tau_{[S,T],H^\rho,L^m}}.
\end{align*}
At the endpoint $\tau=0$ the operator $A$ no longer plays a role and we set
\begin{equ}
    \|Y\|_{\cC^0_{[S,T],H^\rho,L^m}}:=\sup_{u\in[S,T]}\big\|\|Y_u\|_{H^\rho}\big\|_{L^m(\Omega)}.
\end{equ}

Before we proceed to the derivation of the main integral estimates, we introduce additional parameters $\rho$, $\tau$ satisfying the following condition.
\begin{assumption}\label{a:manyineq}
Suppose that $\rho\in[0,\gamma]\cap [0,2)$ and  $\tau \in (\rho/2,1]$ satisfy the following inequalities:
\begin{align}
1-\frac12 (1+\gamma-\rho)(1-\alpha)>\max\Big(\frac12,\tau,\frac{\rho}{2}\Big)\label{eqA}\\
\begin{cases}
    2-\frac{(1+\gamma)(2-\alpha)}{2}+\frac\rho2>1&\text{if }\gamma\in[0,1]
\\
1+\alpha-\frac{1+\gamma}{2}+\frac\rho2>1
&\text{if }\gamma>1
\end{cases}\label{eqB}\\
1-\frac12 (1+\gamma)(1-\alpha)+\tau>1\label{eqC}.
\end{align}

\end{assumption}

Fix now a filtration $(\F_t)_{t\in[0,1]}$ such that $W$ is an $( (\F_t)_{t\in[0,1]})$-cylindrical Brownian motion.

\begin{lemma} \label{lem:comparison}
Let \cref{ass:A,ass:lambdagamma,a:manyineq} hold. Let $m \geqslant 2$, $\alpha \in (0,1]$ and $\theta\ge0$. 
Then there exists  $C=C(\alpha,\gamma,\theta,\rho,\tau,m)>0$ such that for any $(S,T,T_0)\in \Delta^3_{[0,1]}$ and any $H$-valued stochastic processes $\{\phi_t\}_{t \in [0,1]}$, $\{\psi_t\}_{t \in [0,1]}$ adapted to $(\mathcal{F}_t)$
it holds that
\begin{align} \label{eq:comparison2}
&\Big\|\Big\|\int_S^T e^{(T_0-r)\Op}\Big(b(Z_r+\psi_r)-b(Z_r+\phi_r)\Big)\, dr\Big\|_{H^\theta}\Big\|_{L^m(\Omega)}\\
&\qquad\leqslant C [b]_{\cC^\alpha} (T_0-T)^{-\frac\theta2} \Big( (T-S)^{1-\frac{(1+\gamma-\rho)(1-\alpha)}{2}}\|\psi-\phi\|_{\mathcal{C}^0_{[S,T],H,L^m}}^\alpha \|\psi-\phi\|_{\mathcal{C}^0_{[S,T],H^\rho,L^m}}^{1-\alpha}\nonumber\\
&\qquad \quad +\mathbbm{1}_{\gamma\leqslant 1}(T-S)^{2-\frac{(1+\gamma)(2-\alpha)}{2}+\frac\rho2} \|\psi-\phi\|_{\mathcal{C}^0_{[S,T],H^\rho,L^m}}[\psi]_{\cC^{1}_{[S,T],H,L^\infty}}\nonumber\\
&\qquad \quad +\mathbbm{1}_{\gamma>1}(T-S)^{1+\alpha-\frac{1+\gamma}{2}+\frac\rho2} \|\psi-\phi\|_{\mathcal{C}^0_{[S,T],H^\rho,L^m}}[\psi]^\alpha_{\cC^{1}_{[S,T],H,L^\infty}}\nonumber\\
&\qquad \quad +(T-S)^{1-\frac{(1+\gamma)(1-\alpha)}{2}+\tau}[\psi-\phi]^\alpha_{\mathcal{C}^{\tau}_{[S,T],H,L^m}}[\psi-\phi]^{1-\alpha}_{\mathcal{C}^{\tau-\rho/2}_{[S,T],H^\rho,L^m}}\Big).\nonumber
\end{align}
\end{lemma}

\begin{proof}
Note that by linearity it is no loss of generality to assume $[b]_{\cC^\alpha}=1$. 
We aim to apply \cref{lem:SSL}. For $(s,t) \in \Delta^2_{[S,T]}$, set 
\begin{align*}
    A_{s,t}&\coloneqq \EE^s \int_s^t e^{(T_0-r)\Op} b(Z_r+e^{(r-s)\Op}\psi_{s}) - e^{(T_0-r)\Op} b(Z_r+e^{(r-s)\Op}\phi_{s})dr,\\
    \mathcal{A}_{t}&\coloneqq\int_0^t e^{(T_0-r)\Op} b(Z_r+\psi_r)-e^{(T_0-r)\Op}b(Z_r+\phi_r) dr.
\end{align*}
We claim that there exists $C=C(\alpha,\gamma,\theta,\rho,\tau)>0$ such that for any $(s,u,t) \in \Delta^3_{[S,T]}$, we have
\begin{enumerate}[label=(\alph*)]
\item \label{en:a}$$\big\|\|A_{s,t}\|_{H^\theta}\big\|_{L^m(\Omega)}\le   C(T_0-T)^{-\frac\theta2}  (t-s)^{1-\frac{(1+\gamma-\rho)(1-\alpha)}{2}}\big\|\|\psi_s-\phi_s\|_{H}\big\|_{L^m(\Omega)}^\alpha\big\|\|\psi_s-\phi_s\|_{H^\rho}\big\|_{L^m(\Omega)}^{1-\alpha};$$
\item \label{en:b} 
\begin{align*}\big\|&\|\EE^s \delta A_{s,u,t}\|_{H^\theta}\big\|_{L^m(\Omega)}\\
&\le C\mathbbm{1}_{\gamma\le1}(T_0-T)^{-\frac\theta2}  (t-s)^{2-\frac{(1+\gamma)(2-\alpha)}{2}+\frac\rho2} [\psi]_{\mathcal{C}^{1}_{[s,t],H,L^\infty}}\big\|\|\psi_s-\phi_s\|_{H^\rho}\big\|_{L^m(\Omega)}\\
&+ C\mathbbm{1}_{\gamma>1}(T_0-T)^{-\frac\theta2}  (t-s)^{1+\alpha-\frac{1}{2}(1+\gamma)+\frac{\rho}{2}}[\psi]^\alpha_{\mathcal{C}^{1}_{[s,t],H,L^\infty}}\big\|\|\psi_s-\phi_s\|_{H^\rho}\big\|_{L^m(\Omega)}\\
&+ C  (T_0-T)^{-\frac{\theta}{2}}(t-s)^{1-\frac{(1+\gamma)(1-\alpha)}{2}+\tau}[\psi-\phi]^\alpha_{\mathcal{C}^{\tau}_{[s,t],H,L^m}}[\psi-\phi]^{1-\alpha}_{\mathcal{C}^{\tau-\rho/2}_{[s,t],H^\rho,L^m}};
\end{align*}
\item \label{en:c} $\mathcal{A}$ and $A$ satisfy \eqref{eq:SSLconv}, i.e. $\mathcal{A}$ is the limit in probability of the corresponding Riemann sums.
\end{enumerate}
Note that the power of the increment in \ref{en:a} is indeed larger than $\frac12$, and those in \ref{en:b} are larger than $1$, thanks precisely to \cref{a:manyineq}.
Therefore \ref{en:a} verifies \eqref{eq:SSL2} with
\begin{equation*}
\Gamma_1=C(T_0-T)^{-\frac\theta2}\|\psi-\phi\|_{\mathcal{C}^0_{[S,T],H,L^m}}^\alpha \|\psi-\phi\|_{\mathcal{C}^0_{[S,T],H^\rho,L^m}}^{1-\alpha},\quad \eps_1=\frac12-\frac{(1+\gamma-\rho)(1-\alpha)}{2}
\end{equation*}
and \ref{en:b} verifies \eqref{eq:SSL1}
with 
\begin{align*}
&\Gamma_2=C(T_0-T)^{-\frac\theta2}\|\varphi-\psi\|_{\cC^0_{[S,T],H^\rho,L^m}}\Big(\mathbbm{1}_{\gamma\le1}[\psi]_{\cC^{1}_{[S,T],H,L^\infty}}+\mathbbm{1}_{\gamma>1}[\psi]_{\cC^{1}_{[S,T],H,L^\infty}}^\alpha\Big);\\
&\Gamma_3=C(T_0-T)^{-\frac\theta2}[\psi-\varphi]_{\cC^{\tau}_{[S,T],H,L^m}}^\alpha[\psi-\varphi]_{\cC^{\tau-\frac\rho2}_{[S,T],H^\rho,L^m}}^{1-\alpha};\\
&\eps_2=\mathbbm{1}_{\gamma\le1}(1-\frac{(1+\gamma)(2-\alpha)}{2}+\frac\rho2)+
 \mathbbm{1}_{\gamma>1}(\alpha-\frac{1}{2}(1+\gamma)+\frac{\rho}{2});\\
&\eps_3=-\frac12 (1+\gamma)(1-\alpha)+\tau.
\end{align*}
Thus after checking that \ref{en:a}-\ref{en:c} hold true, the claimed bound \eqref{eq:comparison} immediately follows from \cref{lem:SSL}.

We first verify \ref{en:a}. By \cref{p:basic} and \eqref{eq:Holderregtwopoint} in \cref{cor:Holderreg},
\begin{align*}
\big\|\|A_{s,t}\|_{H^\theta}\big\|_{L^m(\Omega)} &\le \int_s^t \Big\|\big\|\EE^s \big(e^{(T_0-r)\Op}b(Z_r+e^{(r-s)\Op}\psi_s)-e^{(T_0-r)\Op}b(Z_r+e^{(r-s)\Op}\phi_s)\big)\big\|_{H^\theta}\Big\|_{L^m(\Omega)} dr \\
&\le C(T_0-T)^{-\frac\theta2}\int_s^t \Big\|\big\|\EE^s \big(b(Z_r+e^{(r-s)\Op}\psi_s)-b(Z_r+e^{(r-s)\Op}\phi_s)\big)\big\|_{H}\Big\|_{L^m(\Omega)} dr \\
&\le  C(T_0-T)^{-\frac\theta2} \int_s^t (r-s)^{\frac{-(1+\gamma-\rho)(1-\alpha)}{2}} \big\|\|\psi_s-\phi_s\|_{H}^\alpha\|\psi_s-\phi_s\|_{H^\rho}^{1-\alpha}\big\|_{L^m(\Omega)}dr\\
&\le   C(T_0-T)^{-\frac\theta2}  (t-s)^{1-\frac{(1+\gamma-\rho)(1-\alpha)}{2}}\big\|\|\psi_s-\phi_s\|_{H}^\alpha\|\psi_s-\phi_s\|_{H^\rho}^{1-\alpha}\big\|_{L^m(\Omega)}\\
&\le   C(T_0-T)^{-\frac\theta2}  (t-s)^{1-\frac{(1+\gamma-\rho)(1-\alpha)}{2}}\big\|\|\psi_s-\phi_s\|_{H}\big\|_{L^m(\Omega)}^\alpha\big\|\psi_s-\phi_s\|_{H^\rho}\big\|_{L^m(\Omega)}^{1-\alpha},
\end{align*}
for $C=C(\alpha,\gamma,\theta,\rho)>0$, where we applied H\"older's inequality in the final inequality.

To prove \ref{en:b}, first note
\begin{align}\label{bstep1}
    \|\EE^s \delta A_{s,u,t}\|_{H^\theta}&=\Big\|\int_u^t \EE^s \Big[e^{(T_0-r)\Op}b(Z_r+e^{(r-s)\Op}\psi_s)-e^{(T_0-r)\Op}b(Z_r+e^{(r-s)\Op}\phi_s)\nn\\
    &\quad -e^{(T_0-r)\Op}b(Z_r+e^{(r-u)\Op}\psi_u)+e^{(T_0-r)\Op}b(Z_r+e^{(r-u)\Op}\phi_u)\Big] dr\Big\|_{H^\theta}.
\end{align}
Adding and subtracting $e^{(T_0-r) \Op } b(Z_r+e^{(r-s)\Op}\phi_s+e^{(r-u)\Op}\psi_u-e^{(r-s)\Op}\psi_s)$, we can split into a four-point error $I_1$ and a two-point error $I_2$. We can control the two-point error  by \cref{p:basic} and \eqref{eq:Holderregtwopoint} in \cref{cor:Holderreg} giving
\begin{align}\label{bstep2}
&\|I_2\|_{L^m(\Omega)}\nn\\
&\,\,\le C(T_0-T)^{-\frac\theta2}  \int_u^t (r-u)^{\frac{-(1+\gamma-\rho)(1-\alpha)}{2}} \big\|\|(\psi-\phi)_u-e^{(u-s)A}(\psi_s-\phi_s)\|_{H}\big\|_{L^m(\Omega)}^{\alpha} \nn\\
&\quad \quad \quad \quad \times \big\|\|(\psi-\phi)_u-e^{(u-s)A}(\psi_s-\phi_s)\|_{H^\rho}\big\|_{L^m(\Omega)}^{1-\alpha} dr\nn\\
&\,\,\le C(T_0-T)^{-\frac\theta2}  \int_u^t (r-u)^{\frac{-(1+\gamma-\rho)(1-\alpha)}{2}} (u-s)^{\tau-\frac{1}{2}\rho(1-\alpha)} [\psi-\phi]^\alpha_{\mathcal{C}^{\tau}_{[s,t],H,L^m}} [\psi-\phi]_{\mathcal{C}^{\tau-\rho/2}_{[s,t],H^\rho,L^m}}^{1-\alpha} dr\nn\\
&\,\,\le C  (T_0-T)^{-\frac{\theta}{2}}(t-s)^{1-\frac{(1+\gamma)(1-\alpha)}{2}+\tau}[\psi-\phi]^\alpha_{\mathcal{C}^{\tau}_{[s,t],H,L^m}}[\psi-\phi]^{1-\alpha}_{\mathcal{C}^{\tau-\rho/2}_{[s,t],H^\rho,L^m}}
\end{align}
for $C=C(\alpha,\gamma,\theta,\rho,\tau)>0$.
Here we used that, by \eqref{eqA}, the exponent $-\frac{(1+\gamma-\rho)(1-\alpha)}{2}$ is larger than $-1$, which ensures integrability in $r$.

To control $I_1$, we distinguish the cases $\gamma\leqslant 1$ and $\gamma>1$. For the former we use \eqref{eq:fourpoint} to get
\begin{align}\label{bstep3}
&\|I_1\|_{L^m(\Omega)}\nn\\
&\,\,\le C (T_0-T)^{-\frac\theta2}\int_u^t (r-u)^{\frac{-(1+\gamma)(2-\alpha)}{2}+\frac\rho2} \big\|\|\psi_u-e^{(u-s)\Op}\psi_s\|_H\big\|_{L^\infty(\Omega)} \big\|\|e^{(u-s)\Op}(\psi_s-\phi_s)\|_{H^\rho}\big\|_{L^m(\Omega)} dr\nn\\
&\,\,\le C (T_0-T)^{-\frac\theta2} \int_u^t (r-u)^{\frac{-(1+\gamma)(2-\alpha)}{2}+\frac\rho2} (u-s) [\psi]_{\mathcal{C}^{1}_{[s,t],H,L^\infty}} \big\|\|\psi_s-\phi_s\|_{H^\rho}\big\|_{L^m(\Omega)}dr\nn\\
&\,\,\le C(T_0-T)^{-\frac\theta2}  (t-s)^{2-\frac{(1+\gamma)(2-\alpha)}{2}+\frac\rho2} [\psi]_{\mathcal{C}^{1}_{[s,t],H,L^\infty}}\big\|\|\psi_s-\phi_s\|_{H^\rho}\big\|_{L^m(\Omega)}
\end{align}
for $C=C(\alpha,\gamma,\theta,\rho,\tau)>0$.
Note that the exponent $\frac{-(1+\gamma)(2-\alpha)}{2}+\frac\rho2$ is larger than $-1$ by \eqref{eqB}, and therefore the integral is finite.  
In the case $\gamma>1$, we use \eqref{eq:modfourpoint} to get 
\begin{align*}
&\|I_1\|_{L^m(\Omega)}\\
&\,\,\le C (T_0-T)^{-\frac\theta2}\int_u^t (r-u)^{-\frac{1}{2}(1+\gamma)+\frac{\rho}{2}} \big\|\|\psi_u-e^{(u-s)\Op}\psi_s\|_H^\alpha\big\|_{L^\infty(\Omega)} \big\|\|e^{(u-s)\Op}(\psi_s-\phi_s)\|_{H^\rho}\big\|_{L^m(\Omega)} dr\\
&\,\,\le C (T_0-T)^{-\frac\theta2} \int_u^t (r-u)^{-\frac{1}{2}(1+\gamma)+\frac{\rho}{2}} (u-s)^\alpha [\psi]^\alpha_{\mathcal{C}^{1}_{[s,t],H,L^\infty}} \big\|\|\psi_s-\phi_s\|_{H^\rho}\big\|_{L^m(\Omega)}dr\\
&\,\,\le C(T_0-T)^{-\frac\theta2}  (t-s)^{1+\alpha-\frac{1}{2}(1+\gamma)+\frac{\rho}{2}}[\psi]^\alpha_{\mathcal{C}^{1}_{[s,t],H,L^\infty}}\big\|\|\psi_s-\phi_s\|_{H^\rho}\big\|_{L^m(\Omega)},
\end{align*}
for $C=C(\alpha,\gamma,\theta,\rho,\tau)>0$, where again in the last inequality we used condition \eqref{eqB} to ensure that the integral converges. Combining this with \eqref{bstep2}, \eqref{bstep3} and substituting into \eqref{bstep1}, we see that claim \ref{en:b} holds. 

It remains to verify \ref{en:c}, that is, we need to show that $\mathcal{A}$ is indeed the limit of the corresponding Riemann sums. Let $t \in [S,T]$. Let $(\Pi_n)_n$ be a sequence of partitions of $[S,t]$ with mesh size going to zero. Denote $\Pi_n=\{t_i^n\}_{i=0}^{N_n}$. First note that we have
\begin{align*}
    \big\|\|\EE^{t_i^n}&[\mathcal{A}_{t_{i+1}^n}-\mathcal{A}_{t_i^n}]-A_{t_i^n,t_{i+1}^n}\|_{H^\theta}\big\|_{L^2(\Omega)}\\
    &\leqslant C (T_0-T)^{-\theta/2}\int_{t_i^n}^{t_{i+1}^n}\Big\|\Big\|e^{(t-r)\Op} b(Z_r+e^{(r-u)\Op}\psi_{u}) - e^{(t-r)\Op} b(Z_r+e^{(r-u)\Op}\phi_{u}) \\
    &\quad- \Big(e^{(t-r)\Op} b(Z_r+\psi_r)-e^{(t-r)\Op}b(Z_r+\phi_r)\Big)\Big\|_H\Big\|_{L^2(\Omega)} dr\\
    &\leqslant C (T_0-T)^{-\theta/2}\|b\|_{\cC^\alpha} (t_{i+1}^n-t_i^n)^{1+\alpha \tau}(1+[\psi]_{\cC^{\tau}_{[S,T],H,L^2}}+[\phi]_{\cC^{\tau}_{[S,T],H,L^2}})
\end{align*}
for $C=C(\theta)>0$.
Then using boundedness of $b$ as well as the above line of inequalities, we get that
\begin{align*}
\Big\|\big\|\mathcal{A}_t-\sum_{i=0}^{N_{n}-1}A_{t_i^n,t_{i+1}^n}\big\|_{H^\theta}\Big\|_{L^2(\Omega)}&\le \Big\|\big\|\sum_{i=0}^{N_{n}-1}(\A_{t_{i+1}^n}-\A_{t_i}^n-\E^{t_i^n}[\A_{t_{i+1}^n}-\A_{t_i}^n])\big\|_{H^\theta}\Big\|_{L^2(\Omega)}\\
&\phantom{\le}+\Big\|\big\|\sum_{i=0}^{N_{n}-1}(\E^{t_i^n}[\A_{t_{i+1}^n}-\A_{t_i}^n]-A_{t_i^n,t_{i+1}^n})\big\|_{H^\theta}\Big\|_{L^2(\Omega)} \\
&\le  
\Big(\sum_{i=0}^{N_{n}-1}\Big\|\big\| \mathcal{A}_{t_{i+1}^n}-\mathcal{A}_{t_i^n}- \EE^{t_i^n}[\mathcal{A}_{t_{i+1}^n}-\mathcal{A}_{t_i^n}]\big\|_{H^\theta}\Big\|^2_{L^2(\Omega)} \Big)^{1/2}\\  &\phantom{\le}+C (T_0-T)^{-\theta/2}\|b\|_{\cC^\alpha} |\Pi_n|^{\alpha \tau}(1+[\psi]_{\cC^{\tau}_{[S,T],H,L^2}}+[\phi]_{\cC^{\tau}_{[S,T],H,L^2}}) \\
&\le C \|b\|_{\C^\alpha} (T_0-T)^{-\theta/2}|\Pi_n|^{\frac12\wedge\alpha\tau}(1+[\psi]_{\cC^{\tau}_{[S,T],H,L^2}}+[\phi]_{\cC^{\tau}_{[S,T],H,L^2}})
\end{align*}
for $C=C(\theta)>0$. Since the right-hand side of the above inequality tends to $0$ as $n\to\infty$, we see that condition \ref{en:c} holds.

Thus, all of the assumptions of \cref{lem:SSL} are satisfied, and the desired bound  \eqref{eq:comparison2} follows from \eqref{sslmres}.
\end{proof}

Combining \cref{e:mbb2,lem:comparison}, we derive the following corollary. 

\begin{corollary} \label{c:comparison}
Let \cref{ass:A,ass:lambdagamma,a:manyineq} hold. Let $m \geqslant 2$ and $\alpha \in (0,1]$. Then there exists
 $C=C(\alpha,\gamma,\rho,\tau,m)>0$ such that for any $(S,T)\in \Delta^2_{[0,1]}$ and any $H$-valued stochastic processes $\{\phi_t\}_{t \in [0,1]}$, $\{\psi_t\}_{t \in [0,1]}$ adapted to $(\mathcal{F}_t)$, it holds that
\begin{align} \label{eq:comparison}
&\Big\|\Big\|\int_S^T e^{(T-r)\Op}\Big(b(Z_r+\psi_r)-b(Z_r+\phi_r)\Big)\, dr\Big\|_{H^\rho}\Big\|_{L^m(\Omega)}\\
&\qquad\leqslant C [b]_{\cC^\alpha}  (T-S)^{1-\frac{(1+\gamma)(1-\alpha)+\alpha \rho}{2}}\|\psi-\phi\|_{\mathcal{C}^0_{[S,T],H,L^m}}^\alpha \|\psi-\phi\|_{\mathcal{C}^0_{[S,T],H^\rho,L^m}}^{1-\alpha}\nonumber\\
&\quad +C [b]_{\cC^\alpha} \mathbbm{1}_{\gamma\leqslant 1} (T-S)^{2-\frac{(1+\gamma)(2-\alpha)}{2}} \|\psi-\phi\|_{\mathcal{C}^0_{[S,T],H^\rho,L^m}}[\psi]_{\cC^{1}_{[S,T],H,L^\infty}}\nonumber\\
&\quad + C [b]_{\cC^\alpha} \mathbbm{1}_{\gamma>1}(T-S)^{1+\alpha-\frac{1+\gamma}{2}}\|\psi-\phi\|_{\mathcal{C}^0_{[S,T],H^\rho,L^m}}[\psi]^\alpha_{\cC^{1}_{[S,T],H,L^\infty}}\nonumber\\
&\quad + C [b]_{\cC^\alpha} (T-S)^{1-\frac{(1+\gamma)(1-\alpha)}{2}+\tau-\frac{\rho}{2}}[\psi-\phi]^\alpha_{\mathcal{C}^{\tau}_{[S,T],H,L^m}}[\psi-\phi]^{1-\alpha}_{\mathcal{C}^{\tau-\rho/2}_{[S,T],H^\rho,L^m}}.\nonumber
\end{align}
\end{corollary}

\begin{proof}
The statement immediately follows from \eqref{eq:comparison2} in \cref{lem:comparison} with $\theta=\rho$ by applying \cref{e:mbb2} to the space $E$ of measurable functions $\Omega\to H^\rho$ equipped with the distance $d(x,y):=\|\|x-y\|_{H^\rho}\|_{L^m(\Omega)}$, and the function
\begin{equation*}
F_s:=\int_0^{T_0-s}  e^{(T_0-r)\Op}\Big(b(Z_r+\psi_r)-b(Z_r+\phi_r)\Big)\, dr,\quad s\in[0,T_0],
\end{equation*}
for $h=3$, $\eta_1=\eta_2=\eta_3=\rho/2$ and the according choices of $\tau_i$.
\end{proof}

Having \cref{lem:comparison,c:comparison} at hand, we are in a position to prove \cref{thm:main}. After a remark on the a priori estimates of solutions, we proceed by first proving uniqueness under the joint assumptions of \cref{lem:comparison} and \cref{c:comparison}. Then, by verifying the regimes of $\gamma$ and $\alpha$ in which $\rho,\tau$ can be chosen such that these assumptions are fulfilled, we prove \cref{thm:main}.

\begin{remark}\label{r:goodbound}
Note that due to boundedness of $b$, we have the following a priori estimates for any (mild) solution $X$ to \eqref{eq:mainSDE}:
\begin{align}
    [X-Z]_{\mathcal{C}^{1}_{[0,1],H,L^\infty}}&\leqslant \|b\|_\infty \label{eq:apriori}\\
    [X-Z]_{\mathcal{C}^{1-\frac{\rho}{2}}_{[0,1],H^\rho,L^\infty}}&\leqslant C(\rho)\|b\|_\infty,\label{eq:apriori2}
\end{align}
where $\rho \in [0,2)$.
To see \eqref{eq:apriori}, note that for any $(s,t) \in \Delta^2_{[0,1]}$, a.s.,
 \begin{align*}
     \|(X-Z)_t-e^{(t-s)\Op} (X-Z)_s\|_{H}&=\Big\|\int_s^t e^{(t-r)\Op}b(X_r) dr\Big\|_{H}
     \leqslant \|b\|_\infty (t-s).
\end{align*}
Inequality \eqref{eq:apriori2} also holds as, by \cref{p:basic},
\begin{align}
\|(X-Z)_t-e^{(t-s)\Op} (X-Z)_s\|_{H^\rho}&=\Big\|\int_s^t e^{(t-r)\Op}b(X_r) dr\Big\|_{H^\rho}\nonumber\\
&\leqslant C(\rho)\int_s^t (t-r)^{-\frac{\rho}{2}} \|b\|_\infty\,dr \nonumber\\
&\leqslant C(\rho) (t-s)^{1-\frac{\rho}{2}}\|b\|_\infty.
\end{align}
\end{remark}

\begin{lemma}\label{lem:exuni}
 Let $x_0\in H$, $\alpha \in (0,1]$, $\gamma\geqslant 0$, $b \in \mathcal{C}^\alpha(H,H)$. Let \cref{ass:A,ass:lambdagamma,a:manyineq} hold.
Then there exists a strong solution to the SDE \eqref{eq:mainSDE} and pathwise uniqueness holds.
\end{lemma}
\begin{proof}
%
\textbf{Pathwise uniqueness.}
To show uniqueness, we aim to apply \cref{lem:comparison,c:comparison}. We first prove uniqueness to \eqref{eq:mainSDE} for small enough $T \in (0,1]$. Then proving uniqueness on $[0,1]$ follows by iterating the argument. 

For ease of exposition, we focus on the case $\gamma\leqslant 1$. At the end of the proof we explain which adaptations have to be made for the case $\gamma>1$.

Let $X,Y$ be two solutions to \eqref{eq:mainSDE} starting from the initial condition $x_0 \in H$. Let $K=X-Z$ and $\wt{K}=Y-Z$; i.e.
\begin{align} \label{eq:KtildeK}
K_t= e^{tA}x_0+\int_0^t e^{(t-r)\Op} b(K_r +Z_r) dr, \quad \wt{K}_t = e^{tA}x_0+\int_0^t e^{(t-r)\Op} b(\wt{K}_r +Z_r) dr.
\end{align} 
Recalling \cref{r:goodbound}, by \cref{lem:comparison} with $\theta=0$, we have that
\begin{align} 
&\Big\|\Big\|\int_s^t e^{(t-r)\Op}\Big(b(Z_r+\wt{K}_r)-b(Z_r+K_r)\Big)\, dr\Big\|_{H}\Big\|_{L^m(\Omega)}\nn\\
&\qquad\leqslant C    (t-s)^{1-\frac12 (1+\gamma-\rho)(1-\alpha)}
\|K-\wt{K}\|_{\cC^0_{[s,t],H,L^m}}^{\alpha}\|K-\wt{K}\|_{\cC^0_{[s,t],H^
\rho,L^m}}^{1-\alpha}\nn\\
&\qquad\phantom{\le}+C  (t-s)^{2-\frac{(1+\gamma)(2-\alpha)}{2}+\frac\rho2}  \|K-\wt{K}\|_{\cC^0_{[s,t],H^\rho,L^m}}\nn\\
&\qquad\phantom{\le}+C(t-s)^{1-\frac12 (1+\gamma)(1-\alpha)+\tau}[\wt{K}-K]_{\mathcal{C}^{\tau}_{[s,t],H,L^m}}^\alpha [\wt{K}-K]_{\mathcal{C}^{\tau-\frac\rho2}_{[s,t],H^\rho,L^m}}^{1-\alpha}\label{cornewm}
\end{align}
for $C=C(\|b\|_{\C^\alpha},\alpha,\gamma,\rho,\tau,m)>0$.
All of the norms and seminorms appearing in the above are finite by \eqref{eq:apriori} and \eqref{eq:apriori2}. Note that, as $K_0=\wt{K}_0=x_0$,
\begin{equation*}
\|K-\wt{K}\|_{\C^0_{[s,t],H,L^m}}\le 
\sup_{r\in[0,T]}\|\|(K_r-\wt{K}_r)-(K_0-\wt{K}_0)\|_H\|_{L^m}\le  [K-\wt{K}]_{\cC^\tau_{[0,T],H,L^m}}
\end{equation*}
and similarly
\begin{equation*}
\|K-\wt{K}\|_{\C^0_{[s,t],H^\rho,L^m}}\le [K-\wt{K}]_{\cC^{\tau-\frac\rho2}_{[0,T],H^\rho,L^m}}.
\end{equation*}
Hence, we get for any $(s,t)\in\Delta_{[0,T]}$,
\begin{align*} 
&\Big\|\Big\|\int_s^t e^{(t-r)\Op}\Big(b(Z_r+\wt{K}_r)-b(Z_r+K_r)\Big)\, dr\Big\|_{H}\Big\|_{L^m(\Omega)}\nn\\
&\qquad\leqslant C    (t-s)^{1-\frac12 (1+\gamma-\rho)(1-\alpha)} [\wt{K}-K]_{\mathcal{C}^{\tau}_{[0,T],H,L^m}}^\alpha [\wt{K}-K]_{\mathcal{C}^{\tau-\frac\rho2}_{[0,T],H^\rho,L^m}}^{1-\alpha}\\
&\qquad\quad +C  (t-s)^{2-\frac{(1+\gamma)(2-\alpha)}{2}+\frac\rho2} \|K-\wt{K}\|_{\cC^0_{[0,T],H^\rho,L^m}},
\end{align*}
where we also used that the power of $(t-s)$ of the third summand in \eqref{cornewm} is larger than the power of the first summand.
Dividing by $(t-s)^\tau$ and taking supremum over $(s,t)\in\Delta_{[0,T]}$, we get
\begin{align} \label{eq:Ctau}
[\wt{K}-K]_{\mathcal{C}^{\tau}_{[0,T],H,L^m}}&\le C T^{1-\frac12 (1+\gamma-\rho)(1-\alpha)-\tau}\|K-\wt{K}\|_{\cC^\tau_{[0,T],H,L^m}}^{\alpha}\|K-\wt{K}\|_{\cC^{\tau-\rho/2}_{[0,T],H^
\rho,L^m}}^{1-\alpha}\nonumber\\
&+C  T^{2-\frac{(1+\gamma)(2-\alpha)}{2}+\frac\rho2-\tau}  \|K-\wt{K}\|_{\cC^0_{[0,T],H^\rho,L^m}}
\end{align}
for $C=C(\|b\|_{\C^\alpha},\alpha,\gamma,\rho,\tau,m)>0$.
Here we used that all the exponents of $T$ are positive thanks to our standing \cref{a:manyineq}.

 One can go through the same argument in $H^\rho$ norm. First, by \cref{c:comparison}, up to multiplying the right hand side of \eqref{cornewm} with $(t-s)^{-\rho/2}$, we get the same bound when taking $H^\rho$-norm instead of $H$-norm on the left hand side. Hence, we get the same bound as \eqref{eq:Ctau} for $[\wt{K}-K]_{\cC^{\tau-\rho/2}_{[0,T]},H^\rho,L^m}$. Indeed, as mentioned above, we lose an extra $\rho/2$ in \eqref{cornewm}, but as we divide by $(t-s)^{\tau-\rho/2}$ in place of $(t-s)^\tau$, we ``gain'' $\rho/2$ back. Thus,
\begin{align} \label{mixed}
[\wt{K}-K]_{\mathcal{C}^{\tau-\frac\rho2}_{[0,T],H^\rho,L^m}}&\le C T^{1-\frac12 (1+\gamma-\rho)(1-\alpha)-\tau}\|K-\wt{K}\|_{\cC^\tau_{[0,T],H,L^m}}^{\alpha}\|K-\wt{K}\|_{\cC^{\tau-\rho/2}_{[0,T],H^
\rho,L^m}}^{1-\alpha}\nn\\
&+C  T^{2-\frac{(1+\gamma)(2-\alpha)}{2}+\frac\rho2-\tau}  \|K-\wt{K}\|_{\cC^0_{[0,T],H^\rho,L^m}}
\end{align}
for $C=C(\|b\|_{\C^\alpha},\alpha,\gamma,\rho,\tau,m)>0$.

Combining \eqref{eq:Ctau} and \eqref{mixed}, we get for $\Lambda:=\|K-\wt{K}\|_{\cC^\tau_{[0,T],H,L^m}}^{\alpha}\|K-\wt{K}\|_{\cC^{\tau-\rho/2}_{[0,T],H^\rho,L^m}}^{1-\alpha}$,
$$
\Lambda\le  C T^{1-\frac12 (1+\gamma-\rho)(1-\alpha)-\tau}\Lambda+C  T^{2-\frac{(1+\gamma)(2-\alpha)}{2}+\frac\rho2-\tau}  \|K-\wt{K}\|_{\cC^0_{[0,T],H^\rho,L^m}}.
$$
Since $\Lambda<\infty$ (recall \cref{r:goodbound}), this gives for $T=T(\|b\|_{\C^\alpha},\alpha,\gamma,\rho,\tau,m)>0$ small enough 
\begin{align}\label{prevnorm}
\Lambda\le C T^{2-\frac{(1+\gamma)(2-\alpha)}{2}+\frac\rho2-\tau}  \|K-\wt{K}\|_{\cC^0_{[0,T],H^\rho,L^m}}.
\end{align}
Substituting this back into \eqref{mixed}, we get 
\begin{align*} 
[\wt{K}-K]_{\mathcal{C}^{\tau-\frac\rho2}_{[0,T],H^\rho,L^m}}&\le C  T^{2-\frac{(1+\gamma)(2-\alpha)}{2}+\frac\rho2-\tau}  \|K-\wt{K}\|_{\cC^0_{[0,T],H^\rho,L^m}}\\
&\leqslant C T^{2-\frac{(1+\gamma)(2-\alpha)}{2}}[\wt{K}-K]_{\mathcal{C}^{\tau-\frac\rho2}_{[0,T],H^\rho,L^m}}
\end{align*}
and thus for $T=T(\|b\|_{\C^\alpha},\alpha,\gamma,\rho,\tau,m)>0$ small enough, 
\begin{align*} 
[\wt{K}-K]_{\mathcal{C}^{\tau-\frac\rho2}_{[0,T],H^\rho,L^m}}=0,
\end{align*}
which completes the proof for $\gamma \leqslant 1$.

The proof for $\gamma>1$ works in precisely the same way; one only has to replace $2-\frac{(1+\gamma)(2-\alpha)}{2}+\frac{\rho}{2}$ with $\alpha+1-\frac{1+\gamma}{2}+\frac{\rho}{2}$ and use the second part of condition \eqref{eqB} to ensure that this exponent is larger than $\tau$.

\textbf{Strong existence:} Again, we only prove the case $\gamma\leqslant 1$. We give a fixed point argument.

For $t\geqslant 0$, $\beta\in(0,1]$, we define $\cL^{\beta,\rho,m}_{[0,t]}$ as the space of stochastic processes $\phi\colon\Omega\times[0,1]\to H$ adapted to the complete filtration generated by $W$ fulfilling $[\phi]_{\cC^{\beta}_{[0,t],H^\rho,L^m}}<\infty$ and a.s. $\phi_0=x_0$. We set up a contraction, for $\tilde{t}>0$ small enough, on 
\begin{equation*}
\mathcal{S}_{\tilde{t}}\coloneqq\{\phi \in \cL^{1,0,\infty}_{[0,\tilde{t}]}\cap \cL^{1-\frac{\rho}{2},\rho,\infty}_{[0,\tilde{t}]}:[\phi]_{\cC^{1}_{[0,\tilde{t}],H,L^\infty}}\leqslant  \|b\|_\infty,[\phi]_{\cC^{1-\frac{\rho}{2}}_{[0,\tilde{t}],H^\rho,L^\infty}}\leqslant \Gamma\|b\|_\infty\},
\end{equation*}
where $\Gamma$ is the constant in \eqref{eq:apriori2}. We consider the above as a closed subspace of 
\begin{equation*}
\Big(\cL^{\tau,0,m}_{[0,\tilde{t}]}\cap \cL^{\tau-\frac{\rho}{2},\rho,m}_{[0,\tilde{t}]},\|\cdot\|_{\cC^{\tau-\frac{\rho}{2}}_{[0,\tilde{t}],H^\rho,L^m}}+\|\cdot\|_{\cC^{\tau}_{[0,\tilde{t}],H,L^m}}\Big).
\end{equation*}
Note that this subspace is nonempty since $\varphi_s=e^{s\Op}x_0$ is an element thereof. Let $\tilde{t} \in [0,T]$.
Then define $\Psi\colon \mathcal{S}_{\tilde{t}}\to \cL^{\tau,0,m}_{[0,\tilde{t}]}\cap \cL^{\tau-\frac{\rho}{2},\rho,m}_{[0,\tilde{t}]}$ by
\begin{align*}
    \Psi(\phi)_t := e^{t\Op}x_0+\int_0^t e^{(t-r)\Op}b(Z_r+\phi_r) \, dr.
\end{align*}
Since, $[\Psi(\phi)]_{\cC^1_{[0,\tilde{t}],H,L^\infty}}\leqslant \|b\|_\infty$ and $[\Psi(\phi)]_{\cC^{1-\rho/2}_{[0,\tilde{t}],H^\rho,L^\infty}}\leqslant \Gamma\|b\|_\infty$, we have
\begin{equation*}
\Psi(\mathcal{S}_{\tilde{t}})\subset \mathcal{S}_{\tilde{t}}.
\end{equation*}
Hence, we can proceed as in the previous part on uniqueness and, using that $\phi_0-\psi_0=0$ a.s., we obtain that
\begin{align*}
[\Psi(\phi)-\Psi(\psi)]_{\mathcal{C}^{\tau}_{[0,\tilde{t}],H,L^m}}&\le C {\tilde{t}}^{1-\frac12 (1+\gamma-\rho)(1-\alpha)-\tau}\|\phi-\psi\|_{\cC^\tau_{[0,\tilde{t}],H,L^m}}^{\alpha}\|\phi-\psi\|_{\cC^{\tau-\rho/2}_{[0,\tilde{t}],H^
\rho,L^m}}^{1-\alpha}\nonumber\\
&\quad +C  {\tilde{t}}^{2-\frac{(1+\gamma)(2-\alpha)}{2}+\frac\rho2-\tau}  \|\phi-\psi\|_{\cC^0_{[0,\tilde{t}],H^\rho,L^m}};\\
[\Psi(\phi)-\Psi(\psi)]_{\mathcal{C}^{\tau-\frac{\rho}{2}}_{[0,\tilde{t}],H^\rho,L^m}}&\le C {\tilde{t}}^{1-\frac12 (1+\gamma-\rho)(1-\alpha)-\tau}\|\phi-\psi\|_{\cC^\tau_{[0,\tilde{t}],H,L^m}}^{\alpha}\|\phi-\psi\|_{\cC^{\tau-\rho/2}_{[0,\tilde{t}],H^
\rho,L^m}}^{1-\alpha}\nonumber\\
&\quad +C  {\tilde{t}}^{2-\frac{(1+\gamma)(2-\alpha)}{2}+\frac\rho2-\tau}  \|\phi-\psi\|_{\cC^0_{[0,\tilde{t}],H^\rho,L^m}}.
\end{align*}
Using that, for $a,b\geqslant 0$, $a^\alpha b^{1-\alpha}\leqslant a+b$, we get
\begin{align*}
    [\Psi(\phi)-\Psi&(\psi)]_{\mathcal{C}^{\tau}_{[0,\tilde{t}],H,L^m}}+[\Psi(\phi)-\Psi(\psi)]_{\mathcal{C}^{\tau-\frac{\rho}{2}}_{[0,\tilde{t}],H^\rho,L^m}}\\
    &\leqslant C {\tilde{t}}^{1-\frac12 (1+\gamma-\rho)(1-\alpha)-\tau}(\|\phi-\psi\|_{\cC^\tau_{[0,\tilde{t}]},H,L^m}+\|\phi-\psi\|_{\cC^{\tau-\rho/2}_{[0,\tilde{t}],H^
\rho,L^m}})\nonumber\\
&\quad +C  {\tilde{t}}^{2-\frac{(1+\gamma)(2-\alpha)}{2}}  [\phi-\psi]_{\cC^{\tau-\frac{\rho}{2}}_{[0,\tilde{t}],H^\rho,L^m}}.
\end{align*}
We see that thanks to \cref{a:manyineq}, all the exponents in the above inequality are positive. Therefore, for $\tilde{t}$ small enough, we get 
\begin{align*}
    [\Psi(\phi)-\Psi(\psi)]_{\mathcal{C}^{\tau}_{[0,\tilde{t}],H,L^m}}+[\Psi(\phi)-\Psi(\psi)]_{\mathcal{C}^{\tau-\frac{\rho}{2}}_{[0,\tilde{t}],H^\rho,L^m}}\leqslant \frac{1}{2}\Big([\phi-\psi]_{\mathcal{C}^{\tau}_{[0,\tilde{t}],H,L^m}}+[\phi-\psi]_{\mathcal{C}^{\tau-\frac{\rho}{2}}_{[0,\tilde{t}],H^\rho,L^m}}\Big)
\end{align*}
and therefore $\Psi$ defines a contraction on $\mathcal{S}_{\tilde{t}}$. Hence, we can apply Banach's fixed point theorem and obtain a fixed point $\phi\in\mathcal{S}_{\tilde{t}}$.
In order to show that we get a solution in the sense of \cref{def:sol}, it remains to find a continuous modification of $\varphi$. This modification is simply given by $\Psi(\varphi)$: Since it is an integral, it is continuous, and since $\Psi(\varphi)=\varphi$ in $\mathcal{S}_{\tilde{t}}$, it is a modification. Therefore, $\Psi(\varphi)+Z$ is a strong solution on $[0,\tilde{t}]$.
Since the point $\tilde{t}$ does not depend on the initial data $x_0$, repeating this argument $\lceil 1/\tilde{t}\rceil$ times, we get a strong solution on the full interval $[0,1]$.
\end{proof}

The purpose of the following lemma is to demonstrate that our main assumption \eqref{eq:gammaupperbound} on the relation of the exponents corresponds precisely to the required inequalities in \cref{a:manyineq}. This allows to use \cref{lem:exuni} to prove \cref{thm:main}.

\begin{lemma}\label{lem:exponents}
    Let $\gamma\geqslant 0$ and $\alpha\in(0,1]$. Then the additional conditions $\gamma<3$ and \eqref{eq:gammaupperbound} are equivalent to the existence of $\tau\in(0,1]$ and $\rho\in(0,\min(\gamma,2\tau])$ such that \eqref{eqA}, \eqref{eqB}, and \eqref{eqC} are satisfied. Moreover, one can also satisfy the inequalities with the additional constraint $\rho<2\tau$.
\end{lemma}
\begin{proof}
    The last statement of the claim is trivial, since all inequalities are strict.
    For the first statement, we start by a few simplifying remarks. Since $\rho\leqslant 2\tau$, $\rho/2$ can be omitted from the maximum in \eqref{eqA}. Furthermore, $\tau\geqslant 1/2$ can be assumed without loss of generality, since increasing $\tau$ does not affect \eqref{eqB}, makes \eqref{eqC} as well as the condition $\rho\leqslant2\tau$ easier to satisfy, and setting $\tau=1/2$ instead of  $\tau<1/2$ also does not affect \eqref{eqA}. Henceforth we restrict our attention to $\tau\in[1/2,1]$ and replace \eqref{eqA} by
    \begin{equ}
        1-\frac12 (1+\gamma-\rho)(1-\alpha)>\tau\label{eqD}.
    \end{equ}
    Finally, at this point increasing $\rho$ makes \eqref{eqB} and \eqref{eqD} easier to satisfy and does not affect \eqref{eqC}, so we can also fix taking the maximal possible choice $\rho=\min(\gamma,2\tau)$.

    First, we consider the case $\gamma\in[0,1]$. By the above remarks, this enforces $\rho=\gamma$, leading to the simpler system
    \begin{equs}
        1-\frac{1-\alpha}{2}>\tau \label{eq:first}
        \\
        2-\frac{(1+\gamma)(2-\alpha)-\gamma}{2}>1 \label{eq:second}
        \\
        1-\frac{(1+\gamma)(1-\alpha)}{2}+\tau>1. \label{eq:third}
    \end{equs}
    The left-hand side of \eqref{eq:second} can be rewritten as $\frac{3}{2}-\frac{(1+\gamma)(1-\alpha)}{2}$, which is  smaller than the left-hand side of \eqref{eq:third} (recall that $\tau\geqslant1/2$), which therefore can be omitted. At this point decreasing $\tau$ is beneficial, leading to setting $\tau=1/2$, which renders that \eqref{eq:first} is automatically fulfilled since $\alpha>0$. We are left with only one inequality, which after rearrangement is seen to be equivalent to $\alpha>\frac{\gamma}{\gamma+1}$, which is exactly \eqref{eq:gammaupperbound} in the regime $\gamma\in[0,1]$.

    We now turn to the case $\gamma>1$, still assuming $\tau \geqslant 1/2$, and discuss the two possibilities $\gamma\lessgtr2\tau$. If we choose $\tau\geqslant\gamma/2$ (which is only possible if $\gamma\leqslant2)$, then as before this implies $\rho=\gamma$, leading to the system
    \begin{equs}
        1-\frac{1-\alpha}{2}>\tau \label{eq:firstalt}
        \\
        1+\alpha-\frac{1}{2}>1 \label{eq:secondalt}
        \\
        1-\frac{(1+\gamma)(1-\alpha)}{2}+\tau>1.\label{eq:thirdalt}
    \end{equs}
    Inequality \eqref{eq:secondalt} simply enforces $\alpha>1/2$. The existence of a $\tau$ that satisfies \eqref{eq:firstalt}, \eqref{eq:thirdalt} as well as $\tau\geqslant\gamma/2$ then requires
    \begin{equ}
        \alpha+1>(1+\gamma)(1-\alpha),\qquad 1+\alpha>\gamma.
    \end{equ}
    The first of these holds as $\alpha>1/2$ and $\gamma\leqslant2$. The second leads to the condition
    \begin{equ}\label{eqX}
        \alpha>\max\Big(\frac{1}{2},\gamma-1\Big).
    \end{equ}
    On the other hand, we may choose $\tau<\gamma/2$, which then implies $\rho=2\tau$ as we fixed $\rho=\min(\gamma,2\tau)$. This leads to the system
    \begin{equs}
        1-\frac{(1+\gamma-2\tau)(1-\alpha)}{2}>\tau
        \\
        1+\alpha-\frac{1+\gamma}{2}+\tau>1
        \\
        1-\frac{(1+\gamma)(1-\alpha)}{2}+\tau>1.
    \end{equs}
    Since we are in the regime $\gamma>1$, the second inequality is a strictly stronger constraint than the third one, which therefore can be omitted.
    Rearranging, we are left to find $\tau$ such that
    \begin{equ}
        \frac{1+\gamma}{2}-\alpha<\tau<\min\Big(\frac{\gamma}{2},\frac{1}{\alpha}\Big(1-\frac{(1+\gamma)(1-\alpha)}{2}\Big)\Big),
    \end{equ}
    which is possible if and only if the leftmost term is smaller than the minimum in the rightmost term.
The first term in the minimum simply leads to the constraint $\alpha>1/2$. The second term, after rearrangement, leads to a quadratic inequality, resulting in the condition
\begin{equ}
    \alpha>\max\Big(\frac{1}{2},\sqrt{\frac{\gamma-1}{2}}\Big),
\end{equ}
which is a more generous condition than \eqref{eqX}, requires $\gamma<3$, and is equivalent to \eqref{eq:gammaupperbound} in the regime $1<\gamma<3$. This finishes the proof.    
\end{proof}

\begin{proof}[Proof of \cref{thm:main}]
Applying \cref{lem:exponents} allows us to choose $\tau$ and $\rho$ satisfying \cref{a:manyineq}. Hence, pathwise uniqueness and existence of a strong solution follows by \cref{lem:exuni}. 

\end{proof}

\bibliographystyle{alpha}
\bibliography{hilbertspace}

@article{Hairer,
  title={An Introduction to {S}tochastic {PDE}s},
  author={Martin Hairer},
  journal={arXiv preprint arXiv:0907.4178},
  year={2009}
}

@article {CG16,
    AUTHOR = {Catellier, R\'emi and Gubinelli, Massimiliano},
     TITLE = {Averaging along irregular curves and regularisation of {ODE}s},
   JOURNAL = {Stochastic Process. Appl.},
  FJOURNAL = {Stochastic Processes and their Applications},
    VOLUME = {126},
      YEAR = {2016},
    NUMBER = {8},
     PAGES = {2323--2366},
      XISSN = {0304-4149,1879-209X},
   XMRCLASS = {34C29 (34F05 60G17 60G22 60H10)},
  XMRNUMBER = {3505229},
XMRREVIEWER = {Fuke\ Wu},
       XDOI = {10.1016/j.spa.2016.02.002},
       XURL = {https://doi.org/10.1016/j.spa.2016.02.002},
}

@article{AddonaBignamini2023,
      title={Pathwise uniqueness for stochastic heat and damped equations with {H}\"older continuous drift}, 
      author={Davide Addona and Davide Augusto Bignamini},
      year={2023},
journal={arXiv preprint arXiv:2308.05415},
}

@article{AddonaBignamini2025,
      title={ Pathwise uniqueness by noise for singular stochastic {PDE}s}, 
    author={Addona, Davide and Bignamini, Davide and Orrieri, Carlo and Scarpa, Luca},
    journal={arXiv preprint arXiv:2512.17736},
  year={2025}
}

@article {DaPratoFlandoli2024,
    AUTHOR = {Da Prato, Giuseppe and Flandoli, Franco},
     TITLE = {Some results for pathwise uniqueness in {H}ilbert spaces},
   JOURNAL = {Commun. Pure Appl. Anal.},
  FJOURNAL = {Communications on Pure and Applied Analysis},
    VOLUME = {13},
      YEAR = {2014},
    NUMBER = {5},
     PAGES = {1789--1797},
      XISSN = {1534-0392},
   MRCLASS = {60H15 (35R60)},
  MRNUMBER = {3230413},
MRREVIEWER = {Paul Andr\'{e} Razafimandimby},
       XDOI = {10.3934/cpaa.2014.13.1789},
       XURL = {https://doi.org/10.3934/cpaa.2014.13.1789},
}

@article {Le,
    AUTHOR = {L{\^e}, Khoa},
     TITLE = {Stochastic sewing in {B}anach spaces},
   JOURNAL = {Electron. J. Probab.},
  FJOURNAL = {Electronic Journal of Probability},
    VOLUME = {28},
      YEAR = {2023},
     PAGES = {Paper No. 26, 22},
   XMRCLASS = {60H99 (46N30 60H50)},
  XMRNUMBER = {4546635},
       XDOI = {10.1214/23-ejp918},
       XURL = {https://doi.org/10.1214/23-ejp918},
}

@article {LasryLions,
    AUTHOR = {Lasry, Jean-Michel and Lions, Pierre-Louis},
     TITLE = {A remark on regularization in {H}ilbert spaces},
   JOURNAL = {Israel J. Math.},
  FJOURNAL = {Israel Journal of Mathematics},
    VOLUME = {55},
      YEAR = {1986},
    NUMBER = {3},
     PAGES = {257--266},
      XISSN = {0021-2172},
   MRCLASS = {41A30 (46N05)},
  MRNUMBER = {876394},
MRREVIEWER = {W. W. Breckner},
      XDOI = {10.1007/BF02765025},
       XURL = {https://doi.org/10.1007/BF02765025},
}

@article {MaslowskiNualart,
    AUTHOR = {Maslowski, Bohdan and Nualart, David},
     TITLE = {Evolution equations driven by a fractional {B}rownian motion},
   JOURNAL = {J. Funct. Anal.},
  FJOURNAL = {Journal of Functional Analysis},
    VOLUME = {202},
      YEAR = {2003},
    NUMBER = {1},
     PAGES = {277--305},
      XISSN = {0022-1236,1096-0783},
   XMRCLASS = {60H10 (34F05 35R60 37L55 60H15)},
  XMRNUMBER = {1994773},
XMRREVIEWER = {M.\ Z\"{a}hle},
       XDOI = {10.1016/S0022-1236(02)00065-4},
       XURL = {https://doi.org/10.1016/S0022-1236(02)00065-4},
}

@article {CoupekMaslowski,
    AUTHOR = {{\v {C}}oupek, Petr and Maslowski, Bohdan},
     TITLE = {Stochastic evolution equations with {V}olterra noise},
   JOURNAL = {Stochastic Process. Appl.},
  FJOURNAL = {Stochastic Processes and their Applications},
    VOLUME = {127},
      YEAR = {2017},
    NUMBER = {3},
     PAGES = {877--900},
      XISSN = {0304-4149,1879-209X},
   XMRCLASS = {60H15 (35R60)},
  XMRNUMBER = {3605714},
XMRREVIEWER = {Jamil\ Abreu},
       XDOI = {10.1016/j.spa.2016.07.003},
       XURL = {https://doi.org/10.1016/j.spa.2016.07.003},
}

@article {KrylovRoeckner,
    AUTHOR = {Krylov, Nicolai Vladimirovich and R\"{o}ckner, Michael},
     TITLE = {Strong solutions of stochastic equations with singular time
              dependent drift},
   JOURNAL = {Probab. Theory Related Fields},
  FJOURNAL = {Probability Theory and Related Fields},
    VOLUME = {131},
      YEAR = {2005},
    NUMBER = {2},
     PAGES = {154--196},
      XISSN = {0178-8051,1432-2064},
   XMRCLASS = {60H20 (35R60 60J60)},
  XMRNUMBER = {2117951},
XMRREVIEWER = {B.\ G.\ Pachpatte},
       XDOI = {10.1007/s00440-004-0361-z},
       XURL = {https://doi.org/10.1007/s00440-004-0361-z},
}

@article{HuWeiYuan,
      title={Stochastic equations with low regularity drifts}, 
      author={Jinlong Wei and Junhao Hu and Chenggui Yuan},
      year={2023},
journal = {arXiv preprint arXiv:2310.00421}, 
}

@article{butkovskygallay,
      title={Weak existence for {SDE}s with singular drifts and fractional {B}rownian or {L}evy noise beyond the subcritical regime}, 
      author={Oleg Butkovsky and Samuel Gallay},
      year={2023},
journal={arXiv preprint arXiv:2311.1201}
}

@article {GaleatiGerencser,
    AUTHOR = {Galeati, Lucio and Gerencs\'{e}r, M\'{a}t\'{e}},
     TITLE = {Solution theory of fractional {SDE}s in complete subcritical
              regimes},
   JOURNAL = {Forum Math. Sigma},
  FJOURNAL = {Forum of Mathematics. Sigma},
    VOLUME = {13},
      YEAR = {2025},
     PAGES = {Paper No. e12, 66},
      XISSN = {2050-5094},
   XMRCLASS = {60H10 (35D30 35Q49 35Q83 35R60 60G22 60H50)},
  XMRNUMBER = {4854429},
XMRREVIEWER = {Marco\ P.\ Cabral},
       XDOI = {10.1017/fms.2024.136},
       XURL = {https://doi.org/10.1017/fms.2024.136},
}

@article {TindelTudorViens,
    AUTHOR = {Tindel, Samy and Tudor, Ciprian A. and Viens, Frederi},
     TITLE = {Stochastic evolution equations with fractional {B}rownian
              motion},
   JOURNAL = {Probab. Theory Related Fields},
  FJOURNAL = {Probability Theory and Related Fields},
    VOLUME = {127},
      YEAR = {2003},
    NUMBER = {2},
     PAGES = {186--204},
      XISSN = {0178-8051,1432-2064},
   XMRCLASS = {60H15 (60G15)},
  XMRNUMBER = {2013981},
XMRREVIEWER = {Ali\ S\"{u}leyman\ \"{U}st\"{u}nel},
       XDOI = {10.1007/s00440-003-0282-2},
       XURL = {https://doi.org/10.1007/s00440-003-0282-2},
}

@article {DuncanMaslowski,
    AUTHOR = {Duncan, Tyrone E. and Pasik-Duncan, Bozenna and Maslowski, Bohdan},
     TITLE = {Fractional {B}rownian motion and stochastic equations in
              {H}ilbert spaces},
   JOURNAL = {Stoch. Dyn.},
  FJOURNAL = {Stochastics and Dynamics},
    VOLUME = {2},
      YEAR = {2002},
    NUMBER = {2},
     PAGES = {225--250},
      XISSN = {0219-4937,1793-6799},
   XMRCLASS = {60H15 (60G15 60H05)},
  XMRNUMBER = {1912142},
XMRREVIEWER = {Anna\ Karczewska},
       XDOI = {10.1142/S0219493702000340},
       XURL = {https://doi.org/10.1142/S0219493702000340},
}

@article {Leinitial,
    AUTHOR = {L{\^e}, Khoa},
     TITLE = {A stochastic sewing lemma and applications},
   JOURNAL = {Electron. J. Probab.},
  FJOURNAL = {Electronic Journal of Probability},
    VOLUME = {25},
      YEAR = {2020},
     PAGES = {Paper No. 38, 55},
      XISSN = {1083-6489},
   XMRCLASS = {60H10 (60H05 60L20)},
  XMRNUMBER = {4089788},
XMRREVIEWER = {Torstein\ K.\ Nilssen},
       XDOI = {10.1214/20-ejp442},
       XURL = {https://doi.org/10.1214/20-ejp442},
}

@article {Priolaerratum,
    AUTHOR = {Priola, Enrico},
     TITLE = {Erratum to ``{A}n optimal regularity result for {K}olmogorov
              equations and weak uniqueness for some critical {SPDE}s''},
   JOURNAL = {Ann. Probab.},
  FJOURNAL = {The Annals of Probability},
    VOLUME = {51},
      YEAR = {2023},
    NUMBER = {6},
     PAGES = {2387--2395},
      XISSN = {0091-1798},
   MRCLASS = {60H15 (35R15 35R60)},
  MRNUMBER = {4666299},
      XDOI = {10.1214/23-aop1650},
       XURL = {https://doi.org/10.1214/23-aop1650},
}

@article {Priola,
    AUTHOR = {Priola, Enrico},
     TITLE = {An optimal regularity result for {K}olmogorov equations and
              weak uniqueness for some critical {SPDE}s},
   JOURNAL = {Ann. Probab.},
  FJOURNAL = {The Annals of Probability},
    VOLUME = {49},
      YEAR = {2021},
    NUMBER = {3},
     PAGES = {1310--1346},
      XISSN = {0091-1798},
   MRCLASS = {60H15 (35R15 35R60)},
  MRNUMBER = {4255146},
MRREVIEWER = {Paul Andr\'{e} Razafimandimby},
      XDOI = {10.1214/20-aop1482},
       XURL = {https://doi.org/10.1214/20-aop1482},
}

@article {AddonaMasieroPriola,
    AUTHOR = {Addona, Davide and Masiero, Federica and Priola, Enrico},
     TITLE = {A {BSDE}s approach to pathwise uniqueness for stochastic
              evolution equations},
   JOURNAL = {J. Differential Equations},
  FJOURNAL = {Journal of Differential Equations},
    VOLUME = {366},
      YEAR = {2023},
     PAGES = {192--248},
      XISSN = {0022-0396},
   MRCLASS = {60H15 (35R60)},
  MRNUMBER = {4582084},
       XDOI = {10.1016/j.jde.2023.04.014},
       XURL = {https://doi.org/10.1016/j.jde.2023.04.014},
}

@article{AddonaBignamini2024,
      title={Pathwise uniqueness in infinite dimension under weak structure conditions}, 
      author={Davide Addona and Davide Augusto Bignamini},
      year={2025},
journal = {arXiv preprint arXiv:2405.14819},
}

@book {Bogachev,
    AUTHOR = {Bogachev, Vladimir I.},
     TITLE = {Gaussian measures},
    SERIES = {Mathematical Surveys and Monographs},
    VOLUME = {62},
 PUBLISHER = {American Mathematical Society, Providence, RI},
      YEAR = {1998},
     PAGES = {xii+433},
      ISBN = {0-8218-1054-5},
   MRCLASS = {60B11 (28C20 46G12 60G15 60H07)},
  MRNUMBER = {1642391},
MRREVIEWER = {Eddy Mayer-Wolf},
      XDOI = {10.1090/surv/062},
       URL = {https://doi.org/10.1090/surv/062},
}

@article {FlandoliDaPrato,
    AUTHOR = {Da Prato, Giuseppe and Flandoli, Franco},
     TITLE = {Pathwise uniqueness for a class of {SDE} in {H}ilbert spaces
              and applications},
   JOURNAL = {J. Funct. Anal.},
  FJOURNAL = {Journal of Functional Analysis},
    VOLUME = {259},
      YEAR = {2010},
    NUMBER = {1},
     PAGES = {243--267},
      XISSN = {0022-1236},
   MRCLASS = {60H10 (46N30)},
  MRNUMBER = {2610386},
MRREVIEWER = {Ruhollah Jahanipur},
      XDOI = {10.1016/j.jfa.2009.11.019},
       XURL = {https://doi.org/10.1016/j.jfa.2009.11.019},
}

@article{ver80,
	author = {Veretennikov, Alexander Yu.},
	fjournal = {Matematicheski\u\i \ Sbornik. Novaya Seriya},
	issn = {0368-8666},
	journal = {Mat. Sb. (N.S.)},
	mrclass = {60H20 (60J65)},
	mrnumber = {568986},
	mrreviewer = {A. B. Buche},
	number = {3},
	pages = {434--452, 480},
	title = {Strong solutions and explicit formulas for solutions of stochastic integral equations},
	volume = {111(153)},
	year = {1980}}

@article{zvonkin74,
	author = {Zvonkin, Alexander K.},
	journal = {Mat. Sb. (N.S.)},
	mrclass = {60H15},
	mrnumber = {0336813},
	mrreviewer = {A. Friedman},
	pages = {129--149, 152},
	title = {A transformation of the phase space of a diffusion process that will remove the drift},
	volume = {93(135)},
	year = {1974}}

@article{D24,
  title={Regularisation by multiplicative noise for reaction-diffusion equations},
  author={Dareiotis, Konstantinos and Holland, Teodor and  L\^{e}, Khoa},
  journal={arXiv preprint arXiv:2409.11130},
  year={2024}}

@article{BM24,
  title={Weak uniqueness for singular stochastic equations},
  author={Butkovsky, Oleg and Mytnik, Leonid},
  journal={arXiv preprint arXiv:2405.13780},
  year={2024}}

@article{ABLM,
	author = {Athreya, Siva and Butkovsky, Oleg and L\^{e}, Khoa and Mytnik, Leonid},
	doi = {10.1002/cpa.22157},
	fjournal = {Communications on Pure and Applied Mathematics},
	issn = {0010-3640,1097-0312},
	journal = {Comm. Pure Appl. Math.},
	mrclass = {35R60 (35K57 60)},
	mrnumber = {4720220},
	number = {5},
	pages = {2708--2777},
	title = {Well-posedness of stochastic heat equation with distributional drift and skew stochastic heat equation},
	url = {https://doi.org/10.1002/cpa.22157},
	volume = {77},
	year = {2024}}

@ARTICLE{HuNualart,
  title     = "Stochastic heat equation driven by fractional noise and local
               time",
  author    = "Hu, Yaozhong and Nualart, David",
  journal   = "Probab. Theory Relat. Fields",
  publisher = "Springer Science and Business Media LLC",
  volume    =  143,
  number    = "1-2",
  pages     = "285--328",
  month     =  jan,
  year      =  2009,}

@article {DFRV16,
    AUTHOR = {Da Prato, Giuseppe and Flandoli, Franco and R\"ockner, Michael and
              Veretennikov, Alexander Yu.},
     TITLE = {Strong uniqueness for {SDE}s in {H}ilbert spaces with
              nonregular drift},
   JOURNAL = {Ann. Probab.},
  FJOURNAL = {The Annals of Probability},
    VOLUME = {44},
      YEAR = {2016},
    NUMBER = {3},
     PAGES = {1985--2023},
      ISSN = {0091-1798,2168-894X},
   MRCLASS = {60H15 (31C25 35R60 60J25)},
  MRNUMBER = {3502599},
MRREVIEWER = {Isamu\ D\^oku},
       DOI = {10.1214/15-AOP1016},
       URL = {https://doi.org/10.1214/15-AOP1016},
}

@article {DFPR13,
    AUTHOR = {Da Prato, Giuseppe and Flandoli, Franco and Priola, Enrico and R\"ockner,
              Michael},
     TITLE = {Strong uniqueness for stochastic evolution equations in
              {H}ilbert spaces perturbed by a bounded measurable drift},
   JOURNAL = {Ann. Probab.},
  FJOURNAL = {The Annals of Probability},
    VOLUME = {41},
      YEAR = {2013},
    NUMBER = {5},
     PAGES = {3306--3344},
      ISSN = {0091-1798,2168-894X},
   MRCLASS = {60H15 (35A02 35R60)},
  MRNUMBER = {3127884},
MRREVIEWER = {Dora\ Sele\v si},
       DOI = {10.1214/12-AOP763},
       URL = {https://doi.org/10.1214/12-AOP763},
}

@article {Cerrai13,
    AUTHOR = {Cerrai, Sandra and Da Prato, Giuseppe and Flandoli, Franco},
     TITLE = {Pathwise uniqueness for stochastic reaction-diffusion
              equations in {B}anach spaces with an {H}\"older drift
              component},
   JOURNAL = {Stoch. Partial Differ. Equ. Anal. Comput.},
  FJOURNAL = {Stochastics and Partial Differential Equations. Analysis and
              Computations},
    VOLUME = {1},
      YEAR = {2013},
    NUMBER = {3},
     PAGES = {507--551},
      ISSN = {2194-0401,2194-041X},
   MRCLASS = {60H15 (35A02 35R60)},
  MRNUMBER = {3327515},
MRREVIEWER = {Anna\ Talarczyk},
       DOI = {10.1007/s40072-013-0016-0},
       URL = {https://doi.org/10.1007/s40072-013-0016-0},
}

@article{BFG,
	author = {Bellingeri, Carlo and Friz, Peter K and Gerencs{\'e}r, M{\'a}t{\'e}},
	journal = {Stochastic Analysis and Applications},
	pages = {1--24},
	publisher = {Taylor \& Francis},
	title = {Singular paths spaces and applications},
	year = {2021}}

@article{le2021taming,
	author = {L{\^e}, Khoa and Ling, Chengcheng},
	journal = {arXiv preprint arXiv:2110.01343},
	title = {Taming singular stochastic differential equations: A numerical method},
	year = {2021}}

\end{document}